\documentclass[a4paper, 10pt]{amsart}

 \usepackage{ifpdf}
 \ifpdf 
     \usepackage[pdftex]{graphicx}   
     \usepackage[pdftex,     
             plainpages=false,   
             breaklinks=true,    
             colorlinks=true,
             pdftitle=My Document
             pdfauthor=My Good Self
            ]{hyperref} 
 \else 
     \usepackage{graphicx}       
     \usepackage{hyperref}       
 \fi 

\usepackage{amsfonts,amsmath}	
\usepackage{amssymb}
\usepackage{verbatim}
\usepackage{amsopn}
\usepackage[english]{babel}
\usepackage{amsthm}
\usepackage{enumerate}
\usepackage{mathrsfs}	
\usepackage{mathtools}
\usepackage{color}

\newcounter{assu}

\theoremstyle{definition} \newtheorem{definition}{Definition}[section]
\theoremstyle{definition} \newtheorem{remark}[definition]{Remark}
\theoremstyle{plain} \newtheorem{lemma}[definition]{Lemma}
\theoremstyle{plain} \newtheorem{proposition}[definition]{Proposition}
\theoremstyle{plain} \newtheorem{theorem}[definition]{Theorem}
\theoremstyle{plain} \newtheorem{corollary}[definition]{Corollary}
\theoremstyle{definition} \newtheorem{example}[definition]{Example}
\theoremstyle{plain} 

\newtheorem*{property}{Property (D)}

\newtheorem*{theorem2}{Theorem \ref{W_main_thm}}

\newtheorem{theorem02}[assu]{Theorem}

\DeclareMathOperator{\conv}{conv}
\DeclareMathOperator{\conc}{conc}
\DeclareMathOperator{\card}{card}
\DeclareMathOperator{\sign}{sign}
\newcommand{\R}{\mathbb{R}}

\newcommand{\N}{\mathbb{N}}
\newcommand{\Z}{\mathbb{Z}}
\newcommand{\TV}{\text{\rm Tot.Var.}}

\newcommand{\W}{\mathcal{W}}
\newcommand{\C}{\mathcal{C}}
\newcommand{\e}{\varepsilon}
\newcommand{\const}{\mathcal{O}(1)}
\newcommand{\fQ}{\mathfrak{Q}}
\newcommand{\sigmarh}{\sigma^{\mathrm{rh}}}
\newcommand{\sigmaent}{\sigma^{\mathrm{ent}}}
\newcommand{\freal}{\mathtt f^{\mathrm{eff}}}
\newcommand{\Qtrans}{Q^{\mathrm{trans}}}
\newcommand{\Lip}{\mathrm{Lip}}

\numberwithin{equation}{section} 

\author{Stefano Bianchini and Stefano Modena}
\title{Quadratic interaction functional for systems of conservation laws: a case study}
\date{\today}


\begin{document}

\begin{abstract}
We prove a quadratic interaction estimate for wavefront approximate solutions to the triangular system of conservation laws
\[
\left\{ \begin{array}{c}
        u_t + \tilde f(u,v)_x = 0, \\
        v_t - v_x = 0.
        \end{array} \right.
\]
This quadratic estimate has been used in the literature to prove the convergence rate of the Glimm scheme \cite{anc_mar_11_CMP}.

Our aim is to extend the analysis, done for scalar conservation laws \cite{bia_mod_13}, in the presence of transversal interactions among wavefronts of different families.
The proof is based on the introduction of a quadratic functional $\mathfrak Q(t)$, decreasing at every interaction, and such that its total variation in time is bounded. 

The study of this particular system is a key step in the proof of the quadratic interaction estimate for general systems: it requires a deep analysis of the wave structure of the solution $(u(t,x),v(t,x))$ and the reconstruction of the past history of each wavefront involved in an interaction.
%
\end{abstract}


\maketitle

\centerline{Preprint SISSA 53/2013/MATE}

\tableofcontents

\section{Introduction}
\label{S_introduction}

Consider a hyperbolic system of conservation laws
\begin{equation}
\label{cauchy}
\left\{
\begin{array}{lll}
u_t + f(u)_x & = & 0 \\
u(0,x) &       = & \bar{u}(x)
\end{array}
\right.
\end{equation}
where $\bar u \in BV(\R,\R^n)$, $f: \R^n \to \R^n$ smooth (by \emph{smooth} we mean at least of class $C^3(\R^n,\R^n)$).

Let $u^\e$ be a wavefront solution to \eqref{cauchy} \cite{bre_00}, where $\e$ is a fixed discretization parameter. Let $\{t_j\}_{j=1, \dots, J}$ be the times at which two wavefronts $w_1,w_2$ \emph{meet} or \emph{collide}; each $t_j$ can be an \emph{interaction} time if the wavefronts $w_1,w_2$ belong to the same family and have the same sign; a \emph{cancellation} time, if $w_1,w_2$ belong to the same family and have opposite sign; a \emph{transversal interaction} time if $w_1,w_2$ belong to different families; a \emph{non-physical interaction} time if at least one among $w_1,w_2$ is a non-physical wavefront (for a precise definition see Definition \ref{D_int_canc_points}).

In a series of papers \cite{anc_mar_11_CMP, anc_mar_11_DCDS, hua_jia_yan_10, hua_yang_10} the following estimate has been discussed:
\begin{equation}
\label{E_quadrati_est}
\sum_{t_j \text{ interaction}} \frac{|\sigma(w_1) - \sigma(w_2)| |w_1| |w_2|}{|w_1| + |w_2|} \leq \mathcal O(1) \TV(\bar u)^2.
\end{equation}
In the above formula $w_1,w_2$ are the wavefronts which interact at time $t_j$, $\sigma(w_1)$ (resp. $\sigma(w_2)$) is the speed of the wavefront $w_1$ (resp. $w_2$) and $|w_1|$ (resp. $|w_2|$) is its strength. Here and in the following, by $\const$ we denote a constant which depends only on the flux function $f$.

As it is shown in \cite{anc_mar_11_DCDS, bia_mod_13}, the proofs presented in the above papers contain glitches/missteps, which justified the publication of a new and different proof in \cite{bia_mod_13}.

This last paper \cite{bia_mod_13} considers the simplest case at the level of \eqref{cauchy}, namely the scalar case $u \in \R$, and shows that nevertheless the analysis is quite complicated: in fact, one has to follow the evolution of every elementary component of a wavefront, which we call \emph{wave} (see Section \ref{Ss_main_results} below and also Definition \ref{W_eow}), an idea present also in \cite{anc_mar_11_CMP}. One of the conclusions of the analysis in \cite{bia_mod_13} is that the functional used to obtain the bound \eqref{E_quadrati_est} is non-local in time, a situation very different from the standard Glimm analysis of hyperbolic systems of conservation laws.

In this work we want to study how the same estimate can be proved in the presence of waves of different families. For this aim, we consider the most simple situation, namely the $2 \times 2$ Temple-class triangular system (see \cite{temple_83} for the definition of Temple class systems)
\begin{equation}
\label{E_temple_intro}
\left\{ \begin{array}{rcl}
        u_t + \tilde f(u,v)_x &=& 0, \\
        v_t - v_x &=& 0,
        \end{array} \right.
\end{equation}
with $\frac{\partial \tilde f (0,0)}{\partial u} > - 1$, so that local uniform hyperbolicity is satisfied. Its quasilinear form in the Riemann coordinates is given by
\begin{equation}
\label{E_temple_2_intro}
\left\{ \begin{array}{rcl}
        w_t + \Big( \frac{\partial \tilde f(u,v)}{\partial u} \Big) w_x &=& 0, \\
        v_t - v_x &=& 0,
        \end{array} \right.
\end{equation}
where $u = u(w,v)$ is the Riemann change of coordinates. Being the equation for the first family $v$ linear, it is sufficient to consider the scalar non-autonomous PDE for $w$,
\begin{equation}
\label{E_scalar_nonauto_intro}
w_t + \bigg( \frac{\partial f}{\partial w}(w,v) \bigg) w_x = 0,
\end{equation}
for some smooth ($C^3$-)function $f$ such that $\frac{\partial f (0,0)}{\partial w} > -1$.

The (non-conservative) Riemann solver we consider for \eqref{E_scalar_nonauto_intro} in general will not generate the standard (entropic) wavefront solution of \eqref{E_temple_intro}: in this paper, we prefer to study the quasilinear system \eqref{E_temple_2_intro} in order to focus on the main difficulty, namely the analysis of the transversal interactions. Indeed, the choice of the coordinates $(w,v)$ and of the (non-conservative) Riemann solver simplifies the computations. 

\noindent Using the fact that the transformations $w \mapsto u(w,v)$, $|v| \ll 1$, are uniformly bi-Lipschitz, it is only a matter of additional technicalities to prove that the analysis in the following sections can be repeated for the standard (entropic) wavefront solution of \eqref{E_temple_intro}. This will be addressed in a forthcoming paper concerning general systems \cite{bia_mod_14}.

\subsection{Main result}
\label{Ss_main_results}

The main result of this paper is the proof of estimate \eqref{E_quadrati_est} for the $\e$-wavefront solution $w_\e$ to \eqref{E_scalar_nonauto_intro}: the parameter $\e$ refers to the discretization $f_\e$ of the flux $f$ and to the discretization $(w_\e(0),v_\e(0))$ of the initial data $(w(0),v(0))$, with, as usual,
\begin{equation}
\label{E_TV_approx_e}
\TV(w_\e(0),v_\e(0)) \leq \TV(w(0),v(0)).
\end{equation}
%

In order to state precisely the main theorem of this paper (Theorem \ref{W_main_thm}), as in \cite{bia_mod_13} we need to introduce what we call an \emph{enumeration of waves} in the same spirit as the \emph{partition of waves} considered in \cite{anc_mar_11_CMP}, see also \cite{anc_mar_10}. Roughly speaking, we assign an index $s$ to each piece of a wavefront (i.e. to each elementary discontinuity of size $\e$), and construct two functions $\mathtt x(t,s)$, $\sigma(t,s)$ which give the position and the speed of the wave $s$ at time $t$, respectively.

More precisely, let $w_\e$ be the $\e$-wavefront solution: for definiteness we assume $w_\e$ to be right continuous in space. Consider the set
\[
\mathcal{W} := \bigg\{1,2,\dots, \frac{1}{\e}\TV(w_\e(0)) \bigg\} \subseteq \N,
\]
which will be called the \emph{set of waves}. In Section \ref{Front_Waves} we construct a function
\begin{equation*}
\begin{array}{ccccc}
\mathtt x &:& [0,
+\infty) \times \mathcal W &\to& (-\infty,+\infty] \\
&& (t,s) &\mapsto& \mathtt x(t,s)
\end{array}
\end{equation*}
with the following properties:
\begin{enumerate}
\item the set $\{t: \mathtt x(t,s) < +\infty\}$ is of the form $[0,T(s))$ with $T(s) \in (0,+\infty)$: define $\mathcal W(t)$ as the set
\[
\W(t) := \big\{ s \in \mathcal W \ | \ \mathtt x(t,s) < +\infty \big\};
\]
\item the function $t \mapsto \mathtt x(t,s)$ is Lipschitz and affine between collisions;
\item for $s < s'$ such that $\mathtt x(t,s),\mathtt x(t,s') < +\infty$ it holds
\[
\mathtt x(t,s) \leq \mathtt x(t,s');
\]
\item there exists a time-independent function $\mathcal S(s) \in \{-1,1\}$, the \emph{sign} of the wave $s$, such that
\begin{equation}
\label{E_push_forw}
D_x w_\e(t, \cdot) = \mathtt x(t,\cdot)_\sharp \big( \mathcal S(\cdot) \,
\e \mathtt{count} \llcorner_{\W(t)} \big),
\end{equation}
where $\mathtt{count}\llcorner_{\W(t)}$ is the counting measure on $\W(t) \subseteq \N$.
\end{enumerate}
The last formula means that for all test functions $\phi \in C^1_c(\R,\R)$ it holds
\[
- \int_\R w_\e(t,x) D_x \phi(x) dx = \e \sum_{s \in \W(t)} \phi(\mathtt x(t,s)) \mathcal S(s).
\]
The fact that $\mathtt x(t,s) = +\infty$ is a convention saying that the wave has been removed from the solution $w_\e$ by a cancellation occurring at time $T(s)$.

\noindent Formula \eqref{E_push_forw} and a fairly easy argument, based on the monotonicity properties of the Riemann solver and used in the proof of Lemma \ref{W_lemma_eow}, yield that to each wave $s$ it is associated a unique value $\hat w(s)$ (independent of $t$) by the formula
\[
\hat w(s) =  w(0,-\infty) + \sum_{\substack{p \in \W(t) \\ p \leq s}} \mathcal S(p).
\]
We finally define the \emph{speed function} $\sigma : [0,+\infty) \times \W \to (-\infty, +\infty]$ as follows: 
\begin{equation}
\label{E_Glimm_speed}
\sigma(t,s) := \left\{
\begin{array}{ll}
+\infty & \text{if } \mathtt x(t,s) = +\infty, \\
\Big(\frac{d}{dw}\conv_{[w_\e(t,\mathtt x(t,s)-),w_\e(t,\mathtt x(t,s))]} f_\e\Big)\Big( \big( \hat w(s)-\e,\hat w(s) \big), v_\e(t,\mathtt x(t,s)) \Big) & \text{if } \mathcal{S}(s) = +1, \\ 
\Big(\frac{d}{dw}\conc_{[w_\e(t,\mathtt x(t,s)),w_\e(t,\mathtt x(t,s)-)]}f_\e\Big) \Big( \big( \hat w(s),\hat w(s)+\e \big), v_\e(t,\mathtt x(t,s)) \Big) & \text{if } \mathcal{S}(s) = -1. 
\end{array}
\right.
\end{equation}
We denote by $f_\e(\cdot,v_\e(t,\mathtt x(t,s)))$ the piecewise affine interpolation of $f(\cdot,v_\e(t,\mathtt x(t,s)))$ with grid size $\e$, as a function of $w$ (see Section \ref{Sss_notations} for the precise definition). 
The definition \eqref{E_Glimm_speed} of $\sigma(t,s)$ means, in other words,  that to the wave $s \in \W(t)$ we assign the speed given by the (non-conservative) Riemann solver in $(t,\mathtt x(t,s))$ to the wavefront containing the interval $(\hat w(s)-\e,\hat w(s))$ for $\mathcal S(s) = +1$ or $(\hat w(s),\hat w(s)+\e)$ for $\mathcal S(s) = -1$.

We can now state our theorem. As before, let $t_j$, $j=1,\dots,J$, be the times where two wavefronts meet, i.e. a collision occurs.

%
%
%

\begin{theorem02}
\label{W_main_thm}
The following holds:
\begin{equation}
\label{E_final_est_tria}
\begin{split}
\sum_{j=1}^{J} \sum_{s \in \W(t_j)} |\sigma(t_j, s)& - \sigma(t_{j-1}, s)||s| \\
\leq &~ \bigg[ 3 \|D^2_{ww}f\| +12\log(2)\|D^3_{wwv}f\| \TV(v(0,\cdot))  \bigg] \TV(w(0,\cdot))^2 \\
&~ + \|D^2_{wv} f\| \TV(w(0,\cdot)) \TV(v(0,\cdot)),
\end{split}
\end{equation}
where $|s| := \e$ is \emph{the strength of the wave $s$}.
\end{theorem02}

Notice that, since the r.h.s. of \eqref{E_final_est_tria} is independent of $\e$ (under the assumption \eqref{E_TV_approx_e}), the above theorem provides a uniform estimate of \eqref{E_quadrati_est} for wavefront tracking solution. In fact, a simple computation based on Rankine-Hugoniot condition yields
\[
\frac{( \sigma(w_1) - \sigma(w_2) ) |w_1| |w_2|}{|w_1| + |w_2|} = \frac{1}{2} \Big[ \big( \sigma(w_1) - \sigma(w) \big) |w_1| + \big( \sigma(w) - \sigma(w_2) \big) |w_2| \Big],
\]
in the case of an interaction of the wavefronts $w_1,w_2$ generating the wavefront $w$, with $w_1$ coming from the left and $w_2$ coming from the right.

\subsection{Sketch of the proof}
\label{Sss_sketch_proof}

As observed in \cite{anc_mar_11_DCDS, bia_mod_13}, the study of wave collisions cannot be local in time, but one has to take into account the whole sequence of interactions-cancellations-transversal interactions involving every couple of waves. 

Our approach in this paper follows the ideas of \cite{bia_mod_13}: we construct a quadratic functional $\mathfrak Q$ such that
\begin{enumerate}[(a)]
\item its total variation in time is bounded by $\mathcal O(1) \TV(w(0),v(0))^2$;
\item \label{Point_2_fQ_char_intro} at any interaction involving the wavefronts $w_1,w_2$, it decays at least of the quantity
\[
\frac{|\sigma(w_1) - \sigma(w_2)| |w_1| |w_2|}{|w_1| + |w_2|}.
\]
\end{enumerate}
The functional can increase only when a transversal interaction occurs, but in this case we show that its positive variation is controlled by the decrease of the classical transversal Glimm interaction functional \cite{gli_65}
\[
\Qtrans(t) := \sum_{h=1}^H \sum_{\substack{s \in \W(t) \\ \mathtt x(t,s) < \mathtt x(t,v_h)}} |v_h||s|.
\]
In the above formula, we denote by $\{v_h\}_{1 \leq h \leq H}$ the wavefronts of the first family generated at $t=0$, by $\{|v_h|\}_{1 \leq h \leq H}$ their strengths, and by $\mathtt x(t,v_h)$, $1 \leq h \leq H$, their position at time $t$. Clearly it holds $\mathtt x(t,v_h) = \mathtt x(0, v_h) - t$, and we assume that $\mathtt x(0,v_h) < \mathtt x(0, v_{h+1})$ for each $h$. 

\noindent Being $\Qtrans$ a Lyapunov functional, it follows that
\begin{equation}
\label{E_pos_fQ_intro}
\text{positive total variation of} \ \mathfrak Q(t) \leq \mathcal O(1) \Qtrans(0) \leq \mathcal O(1) \TV(w(0)) \TV(v(0)),
\end{equation}
so that, being by construction $\mathfrak Q(0) \leq \mathcal O(1) \TV(w(0),v(0))^2$, the functional $t \mapsto \mathfrak Q(t)$ has total variation of the order of $\TV(w(0),v(0))^2$. In particular,
\begin{equation}
\label{E_final_estimr}
\text{left hand side of \eqref{E_final_est_tria} at interactions} \leq \text{negative variation of} \ \mathfrak Q \leq \mathcal O(1) \TV(w(0),v(0))^2.
\end{equation}
The estimates \eqref{E_final_est_tria} concerning transversal interactions and cancellations are much easier (and already done in the literature, see \cite{anc_mar_11_CMP,bia_mod_13}), and we present them in Propositions \ref{P_trans_ch_sped}, \ref{W_canc_3}.


As in \cite{bia_mod_13}, $\mathfrak Q(t)$ has the form 
\[
\mathfrak{Q}(t) :=  \sum_{\substack{s,s' \in \W(t) \\ s < s'}} \mathfrak{q}(t, s, s') |s||s'|.
\]
What differs from the analysis in the scalar case is the computation of the \emph{weights} $\mathfrak q(t,s,s')$.

We recall that in the scalar case the computation of $\mathfrak q(t,s,s')$ involves two steps:
\begin{enumerate}
\item \label{Point_1_scal_cas} the definition of the interval $\mathcal I(t,s,s')$, made of all waves which have interacted \emph{both} with $s$ \emph{and} $s'$;
\item \label{Point_2_scal_cas} the computation of an artificial difference in speed of $s, s'$, obtained by solving the Riemann problem in $\mathcal I(t,s,s')$ with the flux of the scalar equation $f$.
\end{enumerate} 
The fundamental fact in the analysis of the scalar case is that
\begin{property}
\label{Propert_D}
The Riemann solution of Point \eqref{Point_2_scal_cas} {\rm divides} waves which are divided in the approximate solution of the Cauchy problem \eqref{cauchy}.
\end{property}
%
%
%
%
As a consequence, two waves which have been separated at a time $t_j > 0$, can join again at some interaction only if this interaction involves waves which have never interacted.

The main difficulty we face in our setting (i.e. in the presence of wavefronts of different families) is that the two properties above are not true any more. This has an impact in the construction of the intervals $\mathcal I(t,s,s')$ (Point \eqref{Point_1_scal_cas} above) and in the definition of the weights $\mathfrak q$, which is now given by an "artificial flux" (Point \eqref{Point_2_scal_cas}).

We now address these points more deeply.

\subsubsection{\texorpdfstring{Definition of the interval $\mathcal I(t,s,s')$}{Definition of the interval I(t,s,s')}}
\label{Sss_defi_I_intro}

The model situation to be considered here is the following: even in the absence of cancellations or interactions involving waves which have never interacted with $s,s'$, the waves $s,s'$ can undergo to a sequence of splittings and interactions due to the presence of the wavefronts of the first family. In this case, the interval $\mathcal I(t,s,s')$ as defined in \cite{bia_mod_13} does not contain the information about their common story: in fact, by the definition given in \cite{bia_mod_13}, Section 3.3, $\mathcal I(t,s,s')$ contains all waves which have interacted both with $s$ and $s'$, but it gives us no information about the transversal interactions in which each wave $p \in \mathcal I(t,s,s')$ has been involved before time $t$.

Hence, it is more natural to compute the interval $\mathcal I(t,s,s')$ starting from the last common splitting point.

The drawback of this definition is that, differently from the scalar case (Lemma 3.17 of \cite{bia_mod_13}), the intervals $\mathcal I(t,s,s')$ and $\mathcal I(t,p,p')$, for $(s,s') \not= (p,p')$, are not in general comparable. However, a fundamental reduction property still holds, Proposition \ref{P_partition_restr}.


\subsubsection{\texorpdfstring{Computation of the weight $\mathfrak q(t,s,s')$}{Computation of the weight q(t,s,s')}}
\label{Sss_q_def_intro}

The other characteristic of systems is that the scalar reduced flux function (see \cite{bia_03}) \emph{depends} (as a function) on the solution. Hence the separation property cited above is certainly not valid, if we use this reduced scalar flux. Indeed, as we said before, two waves which have split can be again approaching due only to transversal interactions, which means that the flux function $f$ of \eqref{E_scalar_nonauto_intro} is not separating them when solving the Riemann problem in $\mathcal I(t,s,s')$.

\noindent In any case, it is not clear which can be a natural flux function to be used to compute the difference in speeds as in Point \eqref{Point_2_scal_cas}, because of the presence of the wavefronts of the first family.

In order to overcome this difficulty and preserve the separation Property (D), we build first the partition $\mathcal P(t,s,s')$ of $\mathcal I(t,s,s')$ as follows: $\mathcal P(t,s,s')$ is the least refined partition such that for all $t' \leq t$, if $p,p'$ are waves in $\mathcal I(t,s,s')$ which are separated at $t'$, then they belong to different elements $\mathcal J,\mathcal J' \in \mathcal P(t,s,s')$, Proposition \ref{P_divise_partizione_implica_divise_realta}. It is fairly easy to see that the elements $\mathcal J$ of the partition $\mathcal P(t,s,s')$ are intervals.

\noindent The weights $\mathfrak q(t,s,s')$ are then constructed recursively by computing at each transversal interaction the worst possible increase in the difference in speed $\pi(t,s,s')[s,s']$, see \eqref{E_max_ddiff_trans} and Lemma \ref{L_estim_interval_of_partition}, and then defining \eqref{W_mathfrak_q}
\begin{equation}
\label{E_fq_intro_def}
\mathfrak{q}(t_j, s, s') :=
\begin{cases}
0 & \text{$s,s'$ joined at time $t_j$ in the real solution}\\
\dfrac{\pi(t_j,s,s')[s,s']}{|\hat w(s') - (\hat w(s)-\mathcal{S}(s) \e)|} & 
\text{$s,s'$ divided at time $t_j$ and already interacted}, \\
 \|D^2_{ww} f\|_{L^\infty}& \text{$s,s'$ never interacted.}
\end{cases}
\end{equation}

Another difference w.r.t. the scalar case is that here we do not increase the weight $\mathfrak q(t,s,s')$ when a cancellation occurs. Even if we do not obtain the sharpest estimate on \eqref{E_final_est_tria}, this choice is sufficient for proving \eqref{E_quadrati_est} and the analysis is certainly simpler.

\vskip .6cm

In Section \ref{Ss_freal}, we introduce an \emph{effective flux function $\freal$}, which is defined up to an affine function by the formula \eqref{E_defin_freal},
\[
\frac{d^2 \freal_{\bar t}}{dw^2}(w) := \frac{\partial^2 f}{\partial w^2}(w, v) \quad \text{for a.e. $w$,}
\]
where $v = v(\bar t, \mathtt x(\bar t,s))$ for any $s$ such that $w \in (\hat w(s) - \e, \hat w(s)]$ in the case of $s$ positive (resp. $w \in [\hat w(s), \hat w(s) +\e)$ in the case of $s$ negative). As observed before, this flux $\freal$ is not useful for computing the weights: in fact, its main use is in the comparison in the difference in Rankine-Hugoniot speed $\sigmarh(\freal,\mathcal J)$ (obtained by Rankine-Hugoniot condition with flux $\freal$ on the element $\mathcal J$ of the partition $\mathcal P(t,s,s')$) with the weights $\pi(t,s,s')$, see \eqref{E_estim_interval_of_partition}.

\noindent In particular, when no transversal wavefronts are present, then, up to a constant independent of $\mathcal J \in \mathcal P(t,s,s')$, $\sigmarh(\freal,\mathcal J)$ corresponds to the Rankine-Hugoniot speed $\sigmarh(f,\mathcal J)$ computed according to the flux $f$ on the interval $\mathcal J$, and hence the weights $\pi(t,s,s')$ yield a control on the speed difference.

An important consequence of the fact that the intervals $\mathcal I(t,s,s')$, $\mathcal I(t,p,p')$ for $(s,s') \not= (p,p')$ are not comparable, is that the reasoning of Theorem 3.23 of \cite{bia_mod_13} (precisely the inequality before (3.16) in Step 4 of the proof) cannot be carried out. \\
However, the separation property of the partitions $\mathcal P(t,s,s')$ allows to divide the pairs of waves $(s,s')$ (involved in an interaction of the wavefronts $\mathcal L$, $\mathcal R$, with $s \in \mathcal L$, $s' \in \mathcal R$) according to the last transversal interaction-cancellation time which splits them (Lemma \ref{L_incastro_2}). The proof of Point \eqref{Point_2_fQ_char_intro} of page \pageref{Point_2_fQ_char_intro} then proceeds by considering a subtree $D$ of $\{1,2,3\}^{<\N}$ and by constructing for each $\alpha \in D$ a subrectangle $\Psi_\alpha = \mathcal L_\alpha \times \mathcal R_\alpha$ of $\mathcal L \times \mathcal R$, such that the following estimate holds:
\[
\begin{split}
\sigmarh(f,\mathcal L_\alpha) - \sigmarh(f,\mathcal R_\alpha) \leq&~ \sum_{\substack{(s,s') \in \Psi_\alpha \\ (s,s') \text{ already} \\ \text{interacted}}} \pi(t_{j-1},s,s')[s,s']|s||s'| \\
&~ + \sum_{\substack{(s,s') \in \Psi_\alpha \\ (s,s') \text{ never} \\ \text{interacted}}} \|D^2_{ww} f\|_{L^\infty} \Big(|\mathcal L | + |\mathcal R| \Big)|s||s'|.
\end{split}
\]
This is done in Lemma \ref{L_estim_phi_zero_psi_alpha} for the elements $\Psi_\alpha$ with $\alpha$ final leaf of the tree $D$.
A standard argument allows to move backward the above estimate to all the elements $\Psi_\alpha$, $\alpha \in D$, obtaining finally
\[
\begin{split}
\sigmarh(f,\mathcal L) - \sigmarh(f,\mathcal R) \leq&~ \sum_{\substack{(s,s') \in \mathcal L \times \mathcal R \\ (s,s') \text{ already} \\ \text{interacted}}} \pi(t_{j-1},s,s')[s,s']|s||s'| \\
&~ + \sum_{\substack{(s,s') \in \mathcal L \times \mathcal R \\ (s,s') \text{ never} \\ \text{interacted}}} \|D^2_{ww} f\|_{L^\infty} \Big(|\mathcal L | + |\mathcal R| \Big)|s||s'|.
\end{split}
\]
Dividing both side by $|\mathcal L| + |\mathcal R|$ and remembering the definition of $\mathfrak q(t,s,s')$ \eqref{E_fq_intro_def}, we obtain the proof of Point \eqref{Point_2_fQ_char_intro}.

\subsection{Structure of the paper}
\label{Ss_structure}

The paper is organized as follows.

Section \ref{S_convex_env} provides some useful results on convex envelopes. Part of these results are already present in the literature, others can be deduced with little effort. We decided to collect them for reader's convenience. Two particular estimates play a key role in the main body of the paper: the dependence of the derivative of the convex envelope of $f$ when $f$ changes (Proposition \ref{P_estim_diff_conv}) and the behavior of the speed assigned to a wave by the solution to Riemann problem $[u^L,u^R]$ when the left state $u^L$ or the right state $u^R$ are varied (Proposition \ref{diff_vel_proporzionale_canc}).

The next two sections contain the main results of the paper. 

In Section \ref{sect_wavefront} we introduce the main tools which are used in the proof of the main theorem, Theorem \ref{W_main_thm}.

\noindent After recalling how the (non-conservative) wavefront approximated solution $w_\e$ is constructed, we begin with the definition of the wave map $\mathtt x$ and of  \emph{enumeration of waves} in Section \ref{Front_Waves}, Definition \ref{W_eow}. This is the triple $(\mathcal W,\mathtt x,\hat w)$, where $\mathtt x$ is the position of the waves $s$ and $\hat w$ is its right state. In Section \ref{W_pswaves} we show that it is possible to construct a function $\mathtt x(t,s)$ such that at any time $(\mathcal W,\mathtt x(t),\hat w)$ is an enumeration of waves, with $\hat w$ independent on $t$.

\noindent The second tool is the definition of the \emph{effective flux function $\freal$}, Section \ref{Ss_freal}, and we list some of its properties.

\noindent Finally in Section \ref{Ss_transv_Q_glim} we recall the definition of transversal Glimm interaction functional $\Qtrans$ and Proposition \ref{P_Qtrans} recalls the two main properties of $\Qtrans$.

Once we have an enumeration of waves, we can start the proof of Theorem \ref{W_main_thm} (Section \ref{section_W_main_thm}). \\
First we study the estimate \eqref{E_final_est_tria} when a single transversal interaction or a cancellation occurs. These estimates are standard (see for example \cite{anc_mar_11_CMP,bia_mod_13}).

\noindent In the case of a transversal interaction, the variation of speed is controlled by the strength of the wavefront of the first family interacting with the solution $w_\e$, and then the l.h.s. of \eqref{E_final_est_tria} is controlled by the decay of $\Qtrans$. The precise estimate is reported in Proposition \ref{P_trans_ch_sped}; Corollary \ref{C_trans_1} completes the estimate \eqref{E_final_est_tria} for the case of transversal interaction times.

\noindent For cancellation times, the variation of speed is controlled by the amount of cancellation, which in turn is bounded by the decay of a first order functional, namely $\TV(w_\e(t,\cdot))$. This is shown in Proposition \ref{W_canc_3}, where the dependence w.r.t. $\TV(w_\e(t,\cdot))$ and $\|D_{ww} f\|_{L^\infty}$ is singled out. Corollary \ref{W_canc_4} concludes the estimate \eqref{E_final_est_tria} for the case of cancellation times.

\noindent The rest of Section \ref{section_W_main_thm} is the construction and analysis of the functional $\mathfrak Q$ described above (Section \ref{Sss_sketch_proof}), in order to prove Proposition \ref{W_thm_interaction}. This proposition proves \eqref{E_final_est_tria} for the case of interaction times, completing the proof of Theorem \ref{W_main_thm}.

\noindent In Section \ref{W_waves_collision} we define the notion of pairs of waves $(s,s')$ which \emph{have never interacted before a fixed time $t$} and pairs of waves which \emph{have already interacted} and, for any pair of waves which have already interacted, we associate an interval of waves $\mathcal I(t,s,s')$ and a partition $\mathcal P(t,s,s')$ of this interval, which in some sense summarize their past common history. In order to overcome the difficulty mentioned at the beginning of Section \ref{Sss_defi_I_intro}, we introduce the \emph{time of last interaction $\mathtt T(t,s,s')$} \eqref{E_ttT_def}, defined as the last time before $t$ such that $s,s'$ have the same position. The computation of $\mathcal I(t,s,s')$ starts from time $\mathtt T(t,s,s')$ (see \eqref{E_calI_def_1}) as well as the construction of the partition $\mathcal P(t,s,s')$. The desired separation properties of $\mathcal P(t,s,s')$ are proved in Propositions \ref{P_divise_partizione_implica_divise_realta} and \ref{P_partition_restr}.

\noindent In Section \ref{Sss_fQ_def} we write down the functional $\mathfrak Q$ in order to conclude the proof of Theorem \ref{W_main_thm}. 

\noindent Then we study separately the behavior of $\mathfrak Q$ at interactions and transversal interactions-cancellations. Theorem \ref{T_decreasing_without_denominator} and Corollary \ref{W_decreasing} in Section \ref{Ss_Q_decrease} prove that the functional $\mathfrak Q$ decreases at least of the quantity \eqref{E_final_est_tria} at a single interaction time.

\noindent Theorem \ref{W_increasing} in Section \ref{Ss_Q_increase} shows that the increase of $\mathfrak Q$ at each transversal interaction time is controlled by the decrease of the transversal Glimm interaction functional $\Qtrans$. The behavior of $\mathfrak Q$ at cancellations is elementary, due to the definition of the weights $\mathfrak q(t,s,s')$, see the end of Section \ref{Sss_fQ_def}. \\
These two facts conclude the proof of Proposition \ref{W_thm_interaction}, as shown in Section \ref{Sss_sketch_proof}, namely estimates \eqref{E_pos_fQ_intro} and \eqref{E_final_estimr}.

\subsection{Notations}
\label{Sss_notations}

For usefulness of the reader, we collect here some notations used in the subsequent sections.

\begin{itemize}
\item $g(u+) = \lim_{u \rightarrow u^+} g(u)$, $g(u-) = \lim_{u \rightarrow u^-} g(u)$;
\item $g'(u-)$ (resp. $g'(u+)$) is the left (resp. right) derivative of $g$ at point $u$;
\item If $g: [a,b] \to \R$, $h: [b,c] \to \R$ are two functions which coincide in $b$, we define the function $g \cup h: [a,c] \to \R$ as
 \begin{equation*}
 g \cup h (x) = 
 \left\{
 \begin{array}{ll}
 g(x) & \text{if } x \in [a,b], \\
 h(x) & \text{if } x \in [b,c].
 \end{array}
 \right.
 \end{equation*}
\item Sometime we will write $\R_x$ instead of $\R$ (resp. $[0,+\infty)_t$ instead of $[0,+\infty)$) to emphasize the symbol of the variables (resp. $x$ or $t$) we refer to.
\item For any $f: \R \to \R$ and for any $\e >0$, the \emph{piecewise affine interpolation of $f$ with grid size $\e$} is the piecewise affine function $f_\e: \R \to \R$ which coincides with $f$ in the points of the form $m\e$, $m \in \Z$. 
\item If $a,b \in \N$, we will denote by $[a,b] := \big\{n \in \N \ \big| \ a \leq n \leq b\big\}$. From the context it will be always clear if $[a,b]$ is an interval of natural number or the usual interval of real number.
\item Given a Lipschitz function $g: E \subseteq \R \longrightarrow \R$, we denote by 
\[
\Lip(g) := \sup_{\substack{u,v \in E \\ u \neq v}} \frac{|g(v) - g(u)|}{|v-u|}
\]
the best Lipschitz constant of $g$.
\end{itemize}

\section{Convex Envelopes}
\label{S_convex_env}

In this section we define the convex envelope of a continuous function $f: \R \to \R$ in an interval $[a,b]$ and we prove some related results. The first section provides some well-known results about convex envelopes, while in the second section we prove some propositions which will be frequently used in the paper.

The aim of this section is to collect the statements we will need in the main part of the paper. In particular, the most important results are Theorem \ref{convex_fundamental_thm}, concerning the regularity of convex envelopes; Proposition \ref{diff_vel_proporzionale_canc}, referring to the behavior of convex envelopes when the interval $[a,b]$ is varied; Proposition \ref{P_estim_diff_conv}, referring to the to the behavior of convex envelopes when the function is varied: these estimates will play a major role for the study of the Riemann problems.

\subsection{Definitions and elementary results}

\begin{definition}
\label{convex_fcn}
Let $f: \R \to \R$ be continuous and $[a,b] \subseteq \R$. We define \emph{the convex envelope of $f$ in the interval $[a,b]$} as
\[
\conv_{[a,b]}f (u) := \sup\bigg\{g(u) \ \Big| \ g: [a,b] \to \R \text{ is convex and } g \leq f\bigg\}.
\]
\end{definition}

A similar definition holds for \emph{the concave envelope of $f$ in the interval $[a,b]$} denoted by $\conc_{[a,b]}f$. All the results we present here for the convex envelope of a continuous function $f$ hold, with the necessary changes, for its concave envelope.

\begin{lemma}
In the same setting of Definition \ref{convex_fcn}, $\conv_{[a,b]}f$ is a convex function and $\conv_{[a,b]}f(u) \leq f(u)$ for each $u \in [a,b]$.  
\end{lemma}

The proof is straightforward.

Adopting the language of Hyperbolic Conservation Laws, we give the next definition.

\begin{definition}
Let $f$ be a continuous function on $\R$, let $[a,b] \subseteq \R$ and consider $\conv_{[a,b]}f$. A \emph{shock interval} of $\conv_{[a,b]}f$ is an open interval $I \subseteq [a,b]$ such that for each $u \in I$, $\conv_{[a,b]}f(u) < f(u)$.

A \emph{maximal shock interval} is a shock interval which is maximal with respect to set inclusion.

A \emph{shock point} is any $u \in [a,b]$ belonging to a shock interval. A \emph{rarefaction point} is any point $u \in [a,b]$ which is not a shock point, i.e. any point such that $\conv_{[a,b]} f(u) = f(u)$.
\end{definition}

Notice that, if $u \in [a,b]$ is a point such that $\conv_{[a,b]}f(u) < f(u)$, then, by  continuity of $f$ and $\conv_{[a,b]}f$, it is possible to find a maximal shock interval $I$ containing $u$.

It is fairly easy to prove the following result.

\begin{proposition}
\label{shock}
Let $f: \R \to \R$ be continuous; let $[a,b] \subseteq \R$. Let $I$ be a shock interval for $\conv_{[a,b]}f$. Then $\conv_{[a,b]}f$ is affine on $I$.
\end{proposition}

The following theorem is classical and provides a description of the regularity of the convex envelope of a given function $f$. For a self contained proof, see Theorem 2.5 of \cite{bia_mod_13}.
 
\begin{theorem}
\label{convex_fundamental_thm}
Let $f$ be a $\mathcal{C}^{1,1}$ function. Then:
\begin{enumerate}
\item \label{convex_fundamental_thm_1} the convex envelope $\conv_{[a,b]} f$ of $f$ in the interval $[a,b]$ is differentiable on $[a,b]$; 
\item \label{convex_fundamental_thm_2} for each rarefaction point $u \in (a,b)$ it holds
\[
\frac{d}{du}f(u) = \frac{d}{du}\conv_{[a,b]}f(u);
\]
\item \label{convex_fundamental_thm_3} $\frac{d}{du}\conv_{[a,b]} f$ is Lipschitz-continuous with Lipschitz constant less or equal than $\Lip(f')$.
\end{enumerate}
\end{theorem}

By "differentiable on $[a,b]$" we mean that it is differentiable on $(a,b)$ in the classical sense and that in $a$ (resp. $b$) the right (resp. the left) derivative exists. 

\subsection{Further estimates}

We now state some useful results about convex envelopes, which we will frequently use in the following sections. 

\begin{proposition}
\label{tocca}
Let $f: \R \to \R$ be continuous and let $a < \bar{u} < b$. If $\conv_{[a,b]}f(\bar{u}) = f(\bar{u})$, then 
\[
\conv_{[a,b]}f = \conv_{[a,\bar u]}f \cup \conv_{[\bar u,b]}f.
\]
\end{proposition}

\begin{proof} See Proposition 2.7 of \cite{bia_mod_13}.
\end{proof}

%
%
%

\begin{corollary}
\label{RP_ridotto}
Let $f: \R \to \R$ be continuous and let $a < \bar{u} < b$. Assume that $\bar{u}$ belongs to a maximal shock interval $(u_1,u_2)$ with respect to $\conv_{[a,b]}f$. Then $\conv_{[a,\bar u]}f|_{[a,u_1]} = \conv_{[a,b]}f|_{[a,u_1]}$.
\end{corollary}

\begin{proof}
It is an easy consequence of Proposition \ref{tocca}, just observing that by maximality of $(u_1, u_2)$, $\conv_{[a,b]}f(u_1) = f(u_1)$.
\end{proof}

\begin{proposition}
\label{vel_aumenta}
Let $f: \R \to \R$ be continuous; let $a < \bar{u} < b$. Then
\begin{enumerate}
 \item $\big(\frac{d}{du}\conv_{[a,\bar{u}]}f\big)(u+) \geq \big(\frac{d}{du}\conv_{[a,b]}f\big)(u+)$ for each $u \in [a, \bar{u})$;
 \item $\big(\frac{d}{du}\conv_{[a,\bar{u}]}f\big)(u-) \geq \big(\frac{d}{du}\conv_{[a,b]}f\big)(u-)$ for each $u \in (a, \bar{u}]$; 
 \item $\big(\frac{d}{du}\conv_{[\bar{u},b]}f\big)(u+) \leq \big(\frac{d}{du}\conv_{[a,b]}f\big)(u+)$ for each $u \in [\bar{u},b)$;
 \item $\big(\frac{d}{du}\conv_{[\bar{u},b]}f\big)(u-) \leq \big(\frac{d}{du}\conv_{[a,b]}f\big)(u-)$ for each $u \in (\bar{u},b]$.
\end{enumerate}
\end{proposition}

The above statement is identical to Proposition 2.9 of \cite{bia_mod_13}, to which we refer for the proof.
%
%
%
%
%

\begin{proposition}
\label{differenza_vel}
Let $f: \R \to \R$ be continuous; let $a < \bar{u} < b$. Then
\begin{enumerate}
\item for each $u_1, u_2 \in [a, \bar{u})$, $u_1 < u_2$,
\begin{align*}
\Big(\frac{d}{du}\conv_{[a,\bar{u}]}f\Big)(u_2+) - \Big(& \frac{d}{du}\conv_{[a,\bar{u}]}f\Big)(u_1+) \\
\geq&~ \Big(\frac{d}{du}\conv_{[a,b]}f\Big)(u_2+) - \Big(\frac{d}{du}\conv_{[a,b]}f\Big)(u_1+);  
\end{align*}
 \item for each $u_1, u_2 \in (a, \bar{u}]$, $u_1 < u_2$, 
\begin{align*}
\Big(\frac{d}{du}\conv_{[a,\bar{u}]}f\Big)(u_2-) - \Big(&\frac{d}{du}\conv_{[a,\bar{u}]}f\Big)(u_1-) \\ \geq&~
\Big(\frac{d}{du}\conv_{[a,b]}f\Big)(u_2-) - \Big(\frac{d}{du}\conv_{[a,b]}f\Big)(u_1-);
\end{align*}
 \item for each $u_1, u_2 \in [\bar{u},b)$, $u_1 < u_2$,
\begin{align*}
\Big(\frac{d}{du}\conv_{[\bar{u},b]}f\Big)(u_2+) -\Big(&\frac{d}{du}\conv_{[\bar{u},b]}f\Big)(u_1+) \\ \geq&~ \Big(\frac{d}{du}\conv_{[a,b]}f\Big)(u_2+) - \Big(\frac{d}{du}\conv_{[a,b]}f\Big)(u_1+);
\end{align*}
 \item for each $u_1, u_2 \in (\bar{u},b]$, $u_1 < u_2$, 
\begin{align*}
\Big(\frac{d}{du}\conv_{[\bar{u},b]}f\Big)(u_2-) -\Big(&\frac{d}{du}\conv_{[\bar{u},b]}f\Big)(u_1-) \\ \geq&~ \Big(\frac{d}{du}\conv_{[a,b]}f\Big)(u_2-) - \Big(\frac{d}{du}\conv_{[a,b]}f\Big)(u_1-).
\end{align*}
\end{enumerate}
\end{proposition}

\begin{proof}
Easy consequence of previous proposition.
\end{proof}

\begin{corollary}
\label{stesso_shock}
Let $f: \R \to \R$ be continuous and let $a < \bar{u} < b$. Let $u_1,u_2 \in [a,\bar u]$, $u_1<u_2$. If $u_1,u_2$ belong to the same shock interval of $\conv_{[a,\bar u]}f$, then they belong to the same shock interval of $\conv_{[a,b]}f$.
\end{corollary}

\begin{proposition}
\label{diff_vel_proporzionale_canc}
Let $f$ be a $C^{1,1}$ function, let $a < \bar u < b$. Then
\begin{equation*}
\bigg(\frac{d}{du}\conv_{[a,\bar u]} f\bigg)(\bar u-) - \bigg(\frac{d}{du} \conv_{[a,b]}f\bigg)(\bar u) 
\leq \Lip(f') (b - \bar u).
\end{equation*}
Moreover, if $f_\e$ is the piecewise affine interpolation of $f$ with grid size $\e$, it holds
\begin{equation*}
\bigg(\frac{d}{du}\conv_{[a,\bar u]} f_\e\bigg)(\bar u -) - \bigg(\frac{d}{du} \conv_{[a,b]}f_\e\bigg)(\bar u-) 
\leq \Lip(f') (b - \bar u).
\end{equation*}
\end{proposition}

\begin{proof} See Proposition 2.15 of \cite{bia_mod_13}.
\end{proof}

\begin{proposition}
\label{P_estim_diff_conv}
Let $f,g : \R \longrightarrow \R$ be $C^1$ functions. Let $a,b \in \R$, $a<b$. It holds
\[
\bigg\| \frac{d}{du} \conv_{[a,b]}f  - \frac{d}{du} \conv_{[a,b]}g \bigg \|_{L^{\infty}} \leq 
\bigg\| \frac{df}{du}  - \frac{dg}{du} \bigg \|_{L^\infty}.
\]
Moreover, if $f_\e, g_\e$ are the piecewise affine interpolation of $f,g$ respectively, with grid size $\e$, then
\[
\bigg\| \frac{d}{du} \conv_{[a,b]}f_\e  - \frac{d}{du} \conv_{[a,b]}g_\e \bigg \|_{L^{\infty}} \leq 
\bigg\| \frac{df}{du}  - \frac{dg}{du} \bigg \|_{L^\infty}.
\]
\end{proposition}
\begin{proof}
Take $\bar u \in [a,b]$. We want to prove
\begin{equation}
\label{E_estim_diff_conv_pointwise}
\bigg| \frac{d}{du} \conv_{[a,b]}f(\bar u)  - \frac{d}{du} \conv_{[a,b]}g(\bar u) \bigg | \leq 
\bigg\| \frac{df}{du}  - \frac{dg}{du} \bigg \|_{L^\infty}.
\end{equation}
Without loss of generality we can assume that $\bar u$ is a rarefaction point for $f$ and a shock point for $g$. Namely if $\bar u$ is a rarefaction point both for $f$ and for $g$, estimate \eqref{E_estim_diff_conv_pointwise} is a direct consequence of Theorem \ref{convex_fundamental_thm}, Point \eqref{convex_fundamental_thm_2}. If $\bar u$ is a shock point both for $f$ and for $g$, we can always find a point $\hat u$ where $\frac{d}{du} \conv_{[a,b]}f(\bar u) = \frac{d}{du} \conv_{[a,b]}f(\hat u)$, $\frac{d}{du} \conv_{[a,b]}g(\bar u) = \frac{d}{du} \conv_{[a,b]}g(\hat u)$ and $u$ is a rarefaction point for either $f$ or $g$.

Hence let us assume that $\bar u$ is a rarefaction point for $f$ and a shock point for $g$. Set
\[
m := \frac{d}{du} \conv_{[a,b]}f(\bar u), \quad \bar m:= \frac{d}{du} \conv_{[a,b]}g(\bar u).
\]
Since $u$ is a rarefaction point for $f$, it holds
\begin{equation}
\label{E_estim_diff_conv_f}
f(\bar u) +m(u-\bar u) \leq f(u) \quad \text{for each } u \in [a,b].
\end{equation}
On the other hand, since $\bar u$ is a shock point for $g$, denoting by $(\bar a, \bar b)$ the maximal shock interval of $g$ which $\bar u$ belongs to, it holds
\begin{equation}
\label{E_estim_diff_conv_g}
\frac{g(\bar b) - g(\bar u)}{\bar b - \bar u} \leq \bar m \leq \frac{g(\bar u) - g(\bar a)}{\bar u - \bar a}.
\end{equation}
Now assume that $m \geq \bar m$. It holds
\begin{equation}
\label{E_estim_diff_conv_pointwise_proof}
\begin{split}
\bigg| \frac{d}{du} \conv_{[a,b]}f(\bar u)  - \frac{d}{du} \conv_{[a,b]}g(\bar u) \bigg | &~ \quad \ = \quad \ m - \bar m \\
&~ \overset{\eqref{E_estim_diff_conv_f}, \eqref{E_estim_diff_conv_g}}{\leq}  \frac{f(\bar b) - f(\bar u)}{\bar b - \bar u} - \frac{g(\bar b) - g(\bar u)}{\bar b - \bar u} \\
&~ \quad \ \leq \quad \ \frac{1}{\bar b - \bar u} \int_{\bar u}^{\bar b} \bigg( \frac{df}{du}(u) - \frac{dg}{du}(u) \bigg)du \\
&~ \quad \ \leq \quad \ \bigg \| \frac{df}{du} - \frac{dg}{du} \bigg \|_{L^\infty}. 
\end{split}
\end{equation}
If $m \leq \bar m$ a similar argument yields \eqref{E_estim_diff_conv_pointwise}.

In the piecewise affine case, inequalities \eqref{E_estim_diff_conv_f}, \eqref{E_estim_diff_conv_g} still hold with $f_\e$ instead of $f$ and $g_\e$ instead of $g$. We can always assume that $\bar u, \bar a, \bar b \in \Z\e$. Since $f_\e = f, g_\e = g$ on $\Z\e$, the chain of inequalities \eqref{E_estim_diff_conv_pointwise_proof} still holds.
\end{proof}

\begin{proposition}
\label{P_conv_affine}
Let $f: \R \longrightarrow  \R$ be continuous. Let $a,b \in \R$. Let $r(u) = mu+q$ be an affine function, $m,q \in \R$. It holds
\[
\conv_{[a,b]} (f + r) = \Big(\conv_{[a,b]}f \Big) +r.
\]
\end{proposition}
\begin{proof}
Let us prove that $\Big(\conv_{[a,b]}f \Big) +r$ is the convex envelope of $f+r$ in the interval $[a,b]$. First observe that $\Big(\conv_{[a,b]}f \Big) +r$ is convex, since it is sum of convex functions. Next, since $\conv_{[a,b]}f \leq f$, then $\Big( \conv_{[a,b]}f \Big) + r \leq f+r$. Finally let $h$ be any convex function such that $h \leq f+r$. This means that $h - r \leq f$. Since $r$ is affine, $h-r$ is convex and so $h-r \leq \conv_{[a,b]} f$, or, in other words, $h \leq \Big(\conv_{[a,b]} f \Big) +r$.
\end{proof}



\section{Preliminary results}
\label{sect_wavefront}

In this section we construct a wavefront solution to a triangular system, and for this solution we introduce the notions of \emph{waves} and the idea of \emph{enumeration of waves} of the solution $u$. We next construct a scalar flux function $\freal$, the \emph{effective flux function}, which is defined by removing the jumps in the first derivative of $f$ due to the waves of the first family $v$. We conclude the section recalling the definition of the \emph{transversal Glimm interaction potential} and its decay properties.

\subsection{Triangular systems of conservation laws: a case study}
\label{Ss_simple_case_0}

The system of conservation laws we consider is of the form
\begin{equation}
\label{E_temple_1}
\left\{ \begin{array}{rcl}
        u_t + \tilde f(u,v)_x &=& 0, \\
        v_t - v_x &=& 0.
        \end{array} \right.
\end{equation}
The function $\tilde f$ is assumed to satisfy
\[
\frac{\partial \tilde f}{\partial u}(0,0) > -1,
\]
so that the system is uniformly hyperbolic in every compact neighborhood of the origin.

It is elementary to verify that the system \eqref{E_temple_1} is of Temple class \cite{temple_83}, in particular it admits a set of Riemann coordinates $(w,v)$ such that its quasilinear form is given by
\begin{equation}
\label{E_temple_2}
\left\{ \begin{array}{rcl}
        w_t + \frac{\partial \tilde f(u,v)}{\partial u} w_x &=& 0, \\
        v_t - v_x &=& 0,
        \end{array} \right.
\end{equation}
where $u = u(w,v)$ is the Riemann change of coordinates.

%
%
%
%
%

\subsection{Wavefront solution}
\label{Ss_wavefront_solution}

%
%

Let us define $f: \R^2 \longrightarrow \R$ be the relation
\[
\frac{\partial \tilde f}{\partial u} (u(w,v),v) = \frac{\partial f}{\partial w} (w,v).
\]
Since the Riemann change of coordinates has the regularity of $D\tilde f$, we can assume that $f$ is a $C^3$ function satisfying 
\begin{enumerate}
\item $\|D^{\alpha} f\|_{L^{\infty}} < \infty$ for any multindex $\alpha$, $|\alpha| \leq 3$;
\item $\frac{\partial f}{\partial w}(w,v) > -1$ in a neighborhood of $(0,0)$.
\end{enumerate}
We will construct a wavefront solution in the coordinates $(w,v)$, by specifying a (non-conservative) Riemann solver.

\begin{remark}
\label{R_non_entropy}
The solution we will construct in general will not correspond to the standard (entropic) wavefront solution of \eqref{E_temple_1}: in this paper, we prefer to study the quasilinear system \eqref{E_temple_2} in order to focus on the main difficulty, namely the analysis of the transversal interactions. Indeed, the choice of the coordinates $(w,v)$ and of the (non-conservative) Riemann solver simplifies the computations. 

Using the fact that the transformations $w \mapsto u(w,v)$, $|v| \ll 1$, are uniformly bi-Lipschitz, it is only a matter of additional technicalities to prove that the analysis in the following sections can be repeated for the standard (entropic) wavefront solution of \eqref{E_temple_1}. This will be addressed in a forthcoming paper concerning general systems \cite{bia_mod_14}.
\end{remark}

For any $\e > 0$ denote by $f_\e(\cdot, v)$ be the piecewise affine interpolation of $f(\cdot,v)$ with grid size $\e$, as a function of $w$.

We define the \emph{approximate Riemann solver} associated to $f$ as follows: the solution of the Riemann problem 
\[
\big((w^-, v^-), (w^+, v^+)\big) \in (\Z\e)^2 \times (\Z\e)^2 \subseteq \R^2 \times \R^2
\]
is given by the function $(t,x) \mapsto (w(t,x), v(t,x))$, where
\[
v(t,x) = 
\begin{cases}
v^-, & \text{if } x < -t, \\
v^+, & \text{if } x \geq -t,
\end{cases}
\]
while $w(t,x)$ is the piecewise constant, right continuous solution of the scalar Riemann problem $(w^-, w^+)$ with flux function $f_\e(\cdot, v^+)$. 

Let $(\bar w_\e, \bar v_\e)$ be an approximation of the initial datum $(\bar w, \bar v)$ (in the sense that $(\bar w_\e, \bar v_\e) \rightarrow (\bar w, \bar v)$ in $L^1$-norm, as $\e \rightarrow 0$) of the Cauchy problem associated to the system \eqref{E_temple_2}, such that $\bar w_\e, \bar v_\e$ have compact support, they take values in the discrete set $\Z\e$, and 
\begin{equation}
\label{bd_su_dato_iniziale}
\TV(\bar w_\e) \leq \TV(\bar w), \ \ \ \ \TV(\bar v_\e) \leq \TV(\bar v).
\end{equation}

Now, by means of the usual wavefront tracking algorithm, one can construct a function
\[
(t,x) \mapsto (w(t,x), v(t,x))
\]
defined for all $t \geq 0$ and for all $x \in \R$, see for example \cite{bai_bre_97}, \cite{bia_00}. It is easy to see that $(w(t,\cdot), v(t,\cdot))$ is right continuous, compacted supported, piecewise constant and takes values in the set $\Z\e \times \Z\e$. 
 

We will call \emph{wavefronts} the piecewise affine discontinuity curves of the function $(t,x) \mapsto (w(t,x), v(t,x))$; in particular the discontinuity curves of $v$ will be called \emph{wavefronts of the first family} (\emph{positive} or \emph{negative} according to the sign of the jump), while those of $w$ will be called \emph{wavefronts of the second family}. This is a standard notation used in hyperbolic conservation laws.

Let $\{(t_j,x_j)\}$, $j \in \{1,2,\dots,J\}$, be the points in the $t,x$-plane where two (or more) wavefronts collide. Let us suppose that $t_j < t_{j+1}$ and for every $j$ exactly two wavefronts meet in $(t_j, x_j)$. This is a standard assumption, achieved by slightly perturbing the wavefront speed. We also set $t_0 := 0$.

\begin{definition}
\label{D_int_canc_points}
For each $j=1,\dots,J$, we will say that $(t_j,x_j)$ is an \emph{interaction point} (or \emph{non transversal interaction point}) if the wavefronts colliding in $(t_j,x_j)$ are of the second family and have the same sign. An interaction point will be called \emph{positive} (resp. \emph{negative}) if wavefronts which collide in $(t_j,x_j)$ are positive (resp. negative). 

Moreover we will say that $(t_j,x_j)$ is a \emph{cancellation point} if the wavefronts which collide in $(t_j,x_j)$ are of the second family and have opposite sign.

Finally we will say that $(t_j,x_j)$ is a \emph{transversal interaction point} if one of the wavefronts which collide in $(t_j,x_j)$ is of the first family and the other one is of the second family.
\end{definition}

Since, by definition of the Riemann solver, wavefronts of the second family are not created nor split at times $t>0$, the three cases above cover all possibilities.

Let us denote by $\{v_h\}_{1 \leq h \leq H}$ the wavefronts of the first family generated at $t=0$. For each $h$, denote by $v_h^-, v_h^+$ respectively the left and the right state of the wavefront $v_h$, denote by $|v_h| := |v_h^+ - v_h^-|$ its strength, and denote by $\mathtt x(t,v_h)$ the position of the wavefront $v_h$ at time $t$. Clearly it holds $\mathtt x(t,v_h) = \mathtt x(0, v_h) - t$. Assume $\mathtt x(0,v_h) < \mathtt x(0, v_{h+1})$ for each $h$. Finally if $(t_j, x_j)$ is a transversal interaction point, denote by $h(j)$ the index of the wavefront of the second family $v_{h(j)}$ involved in the transversal interaction. 


\subsection{Definition of waves (of the second family)}
\label{Front_Waves}

In this section we define the notion of \emph{wave (of the second family)}, the notion of \emph{position of a wave} and the notion of \emph{speed of a wave}. By definition of wavefront solution, for each time $t \geq 0$, $w_\e(t,\cdot)$ is a piecewise constant, compacted supported function, which takes values in the set $\Z\e$. Hence $\TV(w_\e(t,\cdot))$ is an integer multiple of $\e$.

\subsubsection{Enumeration of waves}
\label{Front_eow}

In this section we define the notion of \emph{enumeration of waves} related to a function $w: \R_x \to \R$ of the single variable $x$: in the following sections, $w$ will be the piecewise constant, compacted supported function $w_\e(t, \cdot)$ for fixed time $t$, considered as a function of $x$.

\begin{definition}
\label{W_eow}
Let $w: \R \to \R$, $w \in BV(\R)$, be a piecewise constant, right continuous function, which takes values in the set $\Z\e$. An \emph{enumeration of waves} for the function $w$ is a 3-tuple $(\mathcal{W}, \mathtt x, \hat w)$, where 
\[
\begin{array}{ll}
\mathcal{W} \subseteq \N & \text{is \emph{the set of waves}},\\
\mathtt x: \mathcal{W} \to (-\infty, +\infty] & \text{is \emph{the position function}}, \\
\hat w: \mathcal{W} \to \Z\e & \text{is \emph{the right state function}}, \\
\end{array}
\]
with the following properties:
\begin{enumerate}
 \item the restriction $\mathtt x|_{\mathtt x^{-1}(-\infty, +\infty)}$ takes values only in the set of discontinuity points of $w$;
 \item the restriction $\mathtt x|_{\mathtt x^{-1}(-\infty, +\infty)}$ is increasing; 
 \item for given $x_0 \in \R$, consider $\mathtt x^{-1}(x_0) = \{s \in \mathcal{W} \ | \ \mathtt x(s) = x_0\}$; then it holds: \begin{enumerate}
   \item if $w(x_0-) < w(x_0)$, then $\hat w|_{\mathtt x^{-1}(x_0)}: \mathtt x^{-1}(x_0) \to (w(x_0-), w(x_0)] \cap \Z\e$ is strictly increasing and bijective; 
   \item if $w(x_0-) > w(x_0)$, then $\hat w|_{\mathtt x^{-1}(x_0)}: \mathtt x^{-1}(x_0) \to [w(x_0), w(x_0-)) \cap \Z\e$ is strictly decreasing and bijective; 
   \item if $w(x_0-) = w(x_0)$, then $\mathtt x^{-1}(x_0) = \emptyset$.
 \end{enumerate}
\end{enumerate}
\end{definition}

Given an enumeration of waves as in Definition \ref{W_eow}, we define the \emph{sign of a wave $s \in \W$} with finite position (i.e. such that $\mathtt x(s) < +\infty$) as follows:
\begin{equation}
\label{W_sign}
\mathcal{S}(s) := \sign \Big[w(\mathtt x(s)) - w(\mathtt x(s)-) \Big].
\end{equation}

We immediately present an example of enumeration of waves which will be fundamental in the sequel. 

\begin{example}
\label{W_initial_eow}
Fix $\e>0$ and let $\bar w_\e \in BV(\R)$ be the first component of the approximate initial datum of the Cauchy problem associated to the system \eqref{E_temple_2}, with compact support and taking values in $\Z\e$. The total variation of $\bar w_\e$ is an integer multiple of $\e$. Let 
\[
W: \R \to [0, \TV(\bar w_\e)], \quad x \mapsto W(x) := \TV(\bar w_\e; (-\infty, x]),
\]
be the total variation function. Then define:
\[
\mathcal{W} := \bigg\{ 1,2,\dots, \frac{1}{\e}\TV(\bar w_\e) \bigg\}
\]
and
\[
\mathtt x_0 : \mathcal{W} \to (-\infty, +\infty], \quad s \mapsto \mathtt x_0(s) := \inf \Big\{x \in (-\infty, +\infty] \ \Big| \ \e s \leq W(x) \Big\}. 
\]
Moreover, recalling \eqref{W_sign}, we define 
\begin{equation*}
\hat w: \mathcal{W} \to \R, \quad s \mapsto \hat w(s) := \bar w_\e(\mathtt x_0(s)-) + \mathcal{S}(s)\Big[\e s - W(\mathtt x_0(s)-) \Big ].
\end{equation*}


It is fairly easy to verify that $\mathtt x_0$, $\hat w$ are well defined and that they provide an enumeration of waves, in the sense of Definition \ref{W_eow}.
\end{example}

\subsubsection{Interval of waves}
\label{W_iow}

In this section we define the notion of \emph{interval of waves} and some related notions and we prove some important results about them. 

As in previous section, consider a function $w: \R \longrightarrow \R$, $w \in BV(\R)$, piecewise constant, compacted supported, right continuous, taking values in the set $\Z\e$ and let $(\W, \mathtt x, \hat w)$ be an enumeration of waves for $w$.

\begin{definition}
\label{W_iow_def}
Let $\mathcal{I} \subseteq \W$. We say that $\mathcal{I}$ is an \emph{interval of waves} for the function $w$ and the enumeration of waves $(\W, \mathtt x, \hat w)$ if for any given $s_1, s_2 \in \mathcal{I}$, with $s_1 \leq s_2$, and for any $p \in \W$
\[
s_1 \leq p \leq s_2 \Longrightarrow p \in \mathcal{I}.
\]
We say that an interval of waves $\mathcal{I}$ is \emph{homogeneous} if for each $s,s' \in \mathcal{I}$, $\mathcal{S}(s) = \mathcal{S}(s')$. If waves in $\mathcal{I}$ are positive (resp. negative), we say that $\mathcal{I}$ is a \emph{positive} (resp. \emph{negative}) interval of waves. 
\end{definition}

\begin{proposition}
\label{W_interval_waves}
Let $\mathcal{I} \subseteq \W$ be a positive (resp. negative) interval of waves. Then the restriction of $\hat w$ to $\mathcal{I}$ is strictly increasing (resp. decreasing) and $\bigcup_{s \in \mathcal{I}} (\hat w(s) - \e, \hat w(s)]$ (resp. $\bigcup_{s \in \mathcal{I}} [\hat w(s), \hat w(s)+\e)$) is an interval in $\R$.
\end{proposition}
\begin{proof}
Assume $\mathcal{I}$ is positive, the other case being similar. First we prove that $\hat w$ restricted to $\mathcal{I}$ is increasing. Let $s, s' \in \mathcal{I}$, with $s < s'$. Let $\xi_0 := \mathtt x(s) < \xi_1 < \dots < \xi_K := \mathtt x(s')$ be the discontinuity points of $w$ between $\mathtt x(s)$ and $\mathtt x(s')$. By definition of `interval of waves' and by the fact that each wave in $\mathcal{I}$ is positive, for any $k = 0, 1, \dots, K$, $
\{p \ | \ \mathtt x(p) = \xi_k\}$ contains only positive waves. Thus, by Definition \ref{W_eow} of enumeration of waves, and by the fact that for each $k=0, \dots, K-1$, $w(\xi_k) = w(\xi_{k+1}-)$, the restriction 
\begin{equation}
\label{W_hat_u_interval_waves}
\hat w : \bigcup_{k=0}^{K} \{p \ | \ \mathtt x(p) = \xi_k\} \to (w(\xi_0 -),w(\xi_K)] \cap \Z\e
\end{equation}
is strictly increasing and bijective, and so $\hat w(s) < \hat w(s')$; hence $\hat w|_{\mathcal{I}}$ is strictly increasing.

In order to prove that $\bigcup_{s \in \mathcal{I}} (\hat w(s) - \e, \hat w(s)]$ is a interval in $\R$, it is sufficient to prove the following: for any $s<s'$ in $ \mathcal{I}$ and for any $m \in \Z$ such that $\hat w(s) < m\e \leq \hat w(s')$, there is $p \in \mathcal{I}$, $s < p \leq s'$ such that $\hat w (p) = m\e$. This follows immediately from the fact that the map in \eqref{W_hat_u_interval_waves} is bijective and strictly increasing.
\end{proof}

Let us give the following definitions.
We define \emph{the strength $|\mathcal I|$ of an interval of waves $\mathcal I$} as 
\[
|\mathcal I| := \e \card \mathcal I.
\] 
Let $I \subseteq \R$ be an interval in $\R$, such that $\inf I, \sup I \in \Z\e$; let $g: \R \longrightarrow \R$ be a $C^{1,1}$ function. The quantity
\begin{equation*}
\sigma^{\text{rh}}(g, I) := \frac{g(\sup I)-g(\inf I)}{\sup I - \inf I}
\end{equation*}
is called \emph{the Rankine-Hugoniot speed given to the interval $I$ by the function $g$}. 

\noindent Moreover let $s$ be any positive (resp. negative) wave such that $(\hat w(s) - \e, \hat w(s)) \subseteq I$ (resp. $(\hat w(s), \hat w(s)+\e) \subseteq I$). The quantity 
\[
\sigma^{\text{ent}}(g,I,s) := \frac{d}{du} \conv_I g_\e \Big( (\hat w(s) - \e, \hat w(s) ) \Big)
\]
(resp. $\sigma(g,I,s) := \frac{d}{du} \conc_I g_\e \Big( (\hat w(s), \hat w(s)+\e ) \Big)$) is called \emph{the entropic speed given to the wave $s$ by the Riemann problem $I$ and the function $g$}.

\noindent If $\sigma^{\text{rh}}(g, I) = \sigma^{\text{ent}}(g,I,s)$ for any $s$, we will say that $I$ is \emph{entropic} w.r.t. the function $g$.

\noindent We will say that \emph{the Riemann problem $I$ with flux function $g$ divides $s,s'$} if $\sigmaent(g,I,s) \neq \sigmaent(g,I,s')$.

\begin{remark}
\label{W_artificial_speed_remark}
Let $\mathcal{I}$ be any positive (resp. negative) interval of waves at fixed time $\bar t$. By Proposition \ref{W_interval_waves}, the set $I := \bigcup_{s \in \mathcal{I}} (\hat w(s) -\e, \hat w(s)]$ 
(resp. $I = \bigcup_{s \in \mathcal{I}} [\hat w(s), \hat w(s)+\e)$) is an interval in $\R$. Hence, we will also write $\sigmarh(g,\mathcal I)$ instead of $\sigmarh(g, I)$ and call it the Rankine-Hugoniot speed given to the interval of waves $\mathcal I$ by the function $g$; we will write $\sigmaent(g,\mathcal{I},s)$ instead of $\sigmaent(g,I,s)$ and call it the entropic speed given to the waves $s$ by the Riemann problem $\mathcal{I}$ with flux function $g$; we will say that $\mathcal I$ is entropic if $I$ is; finally we will say that the Riemann problem $\mathcal{I}$ with flux function $g$ divides $s,s'$ if the Riemann problem $I$ with flux function $g$ does.
\end{remark}

\begin{remark}
\label{W_speed_increasing_wrt_waves}
Notice that $\sigmaent$ is always increasing on $\mathcal I$, whatever the sign of $\mathcal I$ is, by the monotonicity properties of the derivatives of the convex/concave envelopes. 
\end{remark}

\begin{definition}
\label{D_order}
Given two interval of waves $\mathcal I_1, \mathcal I_2$, we will write $\mathcal I_1 < \mathcal I_2$ if for any $s_1 \in \mathcal I_1, s_2 \in \mathcal I_2$, $s_1 < s_2$. We will write $\mathcal I_1 \leq \mathcal I_2$ if either $\mathcal I_1 < \mathcal I_2$ or $\mathcal I_1 = \mathcal I_2$.
\end{definition}

\begin{remark}
\label{R_partition_iow}
Given a function $g$ and an homogenous interval of waves $\mathcal I$, we can always partition $\mathcal I$ through the equivalence relation
\[
p \sim p'  \quad \Longleftrightarrow \quad p,p' \text{ are not divided by the Riemann problem $\mathcal I$ with flux function $g_\e$}.
\]
As a consequence of Remark \ref{W_speed_increasing_wrt_waves}, we have that each element of this partition is an entropic interval of waves and the relation $<$ introduced in Definition \ref{D_order} is a total order on the partition.
\end{remark}

\subsubsection{Position and speed of the waves}
\label{W_pswaves}

Consider the Cauchy problem associated to the system \eqref{E_temple_2} and fix $\e >0$; let $(t,x) \mapsto (w_\e(t,x), v_\e(t,x))$ be the piecewise constant wavefront solution, constructed as in Section \ref{Ss_wavefront_solution}.
For the first component of the initial datum $w_\e(0,\cdot)$, consider the enumeration of waves $(\mathcal{W}, \mathtt x_0, \hat w)$ provided in Example \ref{W_initial_eow}; let $\mathcal{S}$ be the sign function defined in \eqref{W_sign} for this enumeration of waves. 

Now our aim is to define two functions
\[
\mathtt x: [0, +\infty)_t \times \mathcal{W} \to \R_x \cup \{+\infty\}, \qquad
\sigma: [0, +\infty)_t \times \mathcal{W} \to [0,1] \cup \{+\infty\},
\]
called \emph{the position at time $t \in [0, +\infty)$ of the wave $s \in \mathcal{W}$} and \emph{the speed at time $t \in [0, +\infty)$ of the wave $s \in \mathcal{W}$}. As one can imagine, we want to describe the path followed by a single wave $s \in \mathcal{W}$ as time goes on and the speed assigned to it by the Riemann problems it meets along the way. Even if there is a slight abuse of notation (in this section $\mathtt x$ depends also on time), we believe that the context will avoid any misunderstanding.

The function $\mathtt x$ is defined by induction, partitioning the time interval $[0, +\infty)$ in the following way
\[
[0, +\infty) = \{0\} \cup (0, t_1] \cup \dots \cup (t_j, t_{j+1}] \cup \dots \cup (t_{J-1}, t_J] \cup (t_J, +\infty).
\]

First of all, for $t = 0$ we set $\mathtt x(0,s) := \mathtt x_0(s)$, where $\mathtt x_0(\cdot)$ is the position function in the enumeration of waves of Example \ref{W_initial_eow}. Clearly $(\mathcal{W}, \mathtt x(0, \cdot), \hat w)$ is an enumeration of waves for the function $w_\e(0, \cdot)$ as a function of $x$ ($\hat w$ being the right state function, as in the example above).

%

Assume to have defined $\mathtt x(t,\cdot)$ for every $t \leq t_j$ and let us define it for $t \in (t_j, t_{j+1}]$ (or $t \in (t_J, +\infty)$). For any $t \leq t_{j}$ set 
\begin{equation}
\label{E_sigma_t_s}
\sigma(t,s) := \sigmaent\bigg(f\Big(\cdot, v(\mathtt x(t,s))\Big), \ \mathtt x^{-1}\Big(\mathtt x(t,s)\Big), \ s\bigg),
\end{equation}
i.e. the Rankine-Hugoniot speed of the wavefront containing $s$.
\noindent For $t < t_{j+1}$ (or $t_J < t < +\infty$) set
\begin{equation*}
\mathtt x(t,s) := \mathtt x(t_j, s) + \sigma(t_j, s)(t-t_j).
\end{equation*}


\noindent For $t=t_{j+1}$ set
\begin{equation*}
\mathtt x(t_{j+1},s) := \mathtt x(t_j,s) +\sigma(t_j,s)(t_{j+1} - t_j)
\end{equation*}
if $\mathtt x(t_j,s) +\sigma(t_j,s)(t_{j+1} - t_j)$ is not the point of interaction/cancellation/transversal interaction $x_{j+1}$; otherwise for the waves $s$ such that $\mathtt x(t_j,s) +\sigma(t_j,s)(t_{j+1} - t_j) =  x_{j+1}$ and
\[
\mathcal{S}(s)w_\e(t_{j+1},x_{j+1}-) \leq \mathcal{S}(s) \hat w(s) - \e \leq \mathcal{S}(s) \hat w(s) \leq \mathcal{S}(s) w_\e(t_{j+1},x_{j+1})\]
(i.e. the ones surviving the possible cancellation in $(t_{j+1},x_{j+1})$)
define
\[
\mathtt x(t_{j+1},s) := \mathtt x(t_j,s) +\sigma(t_j,s)(t_{j+1} - t_j) =  x_{j+1}.
\]
where $\mathcal{S}(s)$ is defined in \eqref{W_sign}, using the enumeration of waves for the initial datum. To the waves $s$ canceled by a possible cancellation in $(t_{j+1},x_{j+1})$ we assign $\mathtt x(t_{j+1},s) := +\infty$.

The following lemma proves that the above procedure produces an enumeration of waves. 

\begin{lemma}
\label{W_lemma_eow}
For any $\bar t \in (t_j, t_{j+1}]$ (resp. $\bar t \in (t_J, +\infty)$), the 3-tuple $(\mathcal{W}, \mathtt x(\bar t, \cdot), \hat w)$ is an enumeration of waves for the piecewise constant function $w_\e(\bar t,\cdot)$.
\end{lemma}

\begin{proof} We prove separately that the Properties (1-3) of Definition \ref{W_eow} are satisfied.

\smallskip
\noindent {\it Proof of Property (1).} By definition of wavefront solution, $\mathtt x(\bar t, \cdot)$ (restricted to the set of waves where it is finite-valued) takes values only in the set of discontinuity points of $w_\e(\bar t, \cdot)$.

\smallskip
\noindent {\it Proof of Property (2).} Let $s < s'$ be two waves and assume that $\mathtt x(\bar t, s), \mathtt x(\bar t, s') < +\infty$. By contradiction, suppose that $\mathtt x(\bar t, s) > \mathtt x(\bar t, s')$.  Since by the inductive assumption at time $t_j$, the 3-tuple $(\mathcal{W}, \mathtt x(t_j, \cdot), \hat w)$ is an enumeration of waves for the function $w_\e(t_j,\cdot)$, it holds $\mathtt x(t_j,s) \leq \mathtt x(t_j,s')$. Two cases arise:
\begin{itemize}
\item If $\mathtt x(t_j, s) = \mathtt x(t_j, s')$, then it must hold $\sigma(t_j, s) > \sigma(t_j, s')$, but this is impossible, due to Remark \ref{W_speed_increasing_wrt_waves} and equality \eqref{E_sigma_t_s}.

\item If $\mathtt x(t_j, s) < \mathtt x(t_j, s')$, then lines $t \mapsto \mathtt x(t_j,s) + \sigma(t_j, s)(t-t_j)$ and $t \mapsto \mathtt x(t_j,s') + \sigma(t_j, s')(t-t_j)$ must intersect at some time $\tau \in (t_j, \bar t)$, but this is impossible, by definition of wavefront solution and times $(t_j)_j$.
\end{itemize}

\smallskip
\noindent {\it Proof of Property (3).}  For $t < t_{j+1}$ or $t=t_{j+1}$ and for discontinuity points $x \neq x_{j+1}$, the third property of an enumeration of waves is straightforward. So let us check the third property only for time $t = t_{j+1}$ and for the discontinuity point $x_{j+1}$. 

Assume first that wavefronts involved in the collision at $(t_{j+1}, x_{j+1})$ are of the second family, i.e. $(t_{j+1}, x_{j+1})$ is an interaction/cancellation point. Fix any time $\tilde t \in (t_j, t_{j+1})$; according to the assumption on binary intersections, you can find two points $\xi_1, \xi_2 \in \R$ such that for any $s$ with $\mathtt x(t_j,s) + \sigma(t_j,s)(t_{j+1} - t_j) = x_{j+1}$, either $\mathtt x(\tilde t,s) = \xi_1$ or $\mathtt x(\tilde t,s) = \xi_2$ and moreover $w_\e(\tilde t, \xi_1-) = w_\e(t_{j+1}, x_{j+1}-)$, $w_\e(\tilde t, \xi_2) = w_\e(t_{j+1}, x_{j+1})$,  $w_\e(\tilde t, \xi_1) = w_\e(\tilde t, \xi_2-)$. 

We now just consider two main cases: the other ones can be treated similarly. Recall that at time $\tilde t < t_{j+1}$, the $3$-tuple $(\W, \mathtt x(\tilde t, \cdot), \hat w)$ is an enumeration of waves for the piecewise constant function $w_\e(\tilde t, \cdot)$.

If $w_\e(\tilde t, \xi_1-) < w_\e(\tilde t, \xi_1) = w_\e(\tilde t,\xi_2-) < w_\e(\tilde t, \xi_2)$, then 
$$\hat w|_{\mathtt x^{-1}(\tilde t,\xi_1)}: \mathtt x^{-1}(\tilde t, \xi_1) \to (w_\e(\tilde t, \xi_1-), w_\e(\tilde t, \xi_1)] \cap \Z\e$$ 
and 
$$\hat w|_{\mathtt x^{-1}(\tilde t,\xi_2)}: \mathtt x^{-1}(\tilde t, \xi_2) \to (w_\e(\tilde t, \xi_2-), w_\e(\tilde t, \xi_2)] \cap \Z\e$$ 
are strictly increasing and bijective; observing that in this case $\mathtt x^{-1}(t_{j+1},x_{j+1}) = \mathtt x^{-1}(\tilde t, \xi_1) \cup \mathtt x^{-1}(\tilde t, \xi_2)$, one gets the thesis.

If $w_\e(\tilde t, \xi_1-) < w_\e(\tilde t, \xi_2) < w_\e(\tilde t, \xi_1) = w_\e(\tilde t, \xi_2-)$, then 
$$\hat w|_{\mathtt x^{-1}(\tilde t,\xi_1)}: \mathtt x^{-1}(\tilde t, \xi_1) \to (w_\e(\tilde t, \xi_1-), w_\e(\tilde t, \xi_1)] \cap \Z\e$$
is strictly increasing and bijective; observing that in this case 
$$\mathtt x^{-1}(t_{j+1},x_{j+1}) = \Big\{s \in \mathtt x^{-1}(\tilde t, \xi_1) \ | \ \hat w(s) \in (w_\e(\tilde t, \xi_1-), w_\e(\tilde t, \xi_2)] \Big\},$$
one gets the thesis.

Now assume that $(t_{j+1}, x_{j+1})$ is a transversal interaction point. In this case, by the definition of the Riemann solver we are using, you can easily find a time $\tilde t < t_{j+1}$ and a point $\tilde \xi \in \R$ such that
\[
\{ s \ | \ \mathtt x(\tilde t, s) = \tilde \xi \} = \{s \ | \ \mathtt x(t_{j+1},s) = x_{j+1} \}
\]
and 
\[
w_\e (\tilde t, \tilde \xi -) = w_\e (t_{j+1}, x_{j+1} -), \quad w_\e (\tilde t, \tilde \xi) = w_\e (t_{j+1}, x_{j+1}).
\]
From the fact that at time $\tilde t$ the $3$-tuple $(\W, \mathtt x(\tilde t, \cdot), \hat w)$ is an enumeration of waves for the function $w_\e (\tilde t, \cdot)$, one gets the thesis.
\end{proof}

\begin{remark}
For fixed wave $s$, $t \mapsto \mathtt x(t, s)$ is Lipschitz, while $t \mapsto \sigma(t,s)$ is right-continuous and piecewise constant.
\end{remark}


To end this section, we introduce the following notations. Given a time $t \in [0, +\infty)$ and a position $x \in (-\infty, +\infty]$, we set
\begin{equation*}
\W(t) :=  \{s \in \W \ | \ \mathtt x(t,s) < +\infty \}, \qquad
\W(t,x) := \{s \in \W \ | \ \mathtt x(t,s) = x \}.
\end{equation*}
We will call $\W(t)$ the set of the \emph{real waves}, while we will say that a wave $s$ is \emph{removed} or \emph{canceled} at time $t$ if $\mathtt x(t,s) = + \infty$. It it natural to define the interval of existence of $s \in \W(0)$ by 
\[
T(s) := \sup \Big\{t \in [0, +\infty) \ | \ \mathtt x(t,s) < +\infty \Big\}.
\]

\noindent If $\mathcal I \subseteq \W(t)$ is an interval of waves for the function $w_\e(t, \cdot)$ and the enumeration of waves $(\W, \mathtt x(t,\cdot), \hat w)$, we will say that $\mathcal I$ is \emph{an interval of waves at time $t$}.

\subsection{\texorpdfstring{The effective flux function $\freal$}{The effective flux function}}
\label{Ss_freal}

As in the previous sections, let $(t,x) \mapsto w_{\e}(t,x)$ be the first component of the $\e$-wavefront solution of the Cauchy problem \eqref{cauchy} constructed before; consider the enumeration of waves and the related position function $(t,s) \mapsto \mathtt x(t,s)$ and speed function $(t,s) \mapsto \sigma(t,s)$ defined in previous section. 

Fix any time $\bar t$. Partitioning $\W(\bar t)$ with respect to the equivalence relation
\begin{equation*}
s \sim s' \quad \Longleftrightarrow \quad [s,s'] \cap \W(\bar t) \text{ is an homogeneous interval of waves},
\end{equation*}
it is possible to write $\W(\bar t)$ as a finite union of mutually disjoint, maximal (with respect to set inclusion) homogenous interval of waves $\mathcal M_l$,
\begin{equation*}
\W(\bar t) = \bigcup_{l = 1}^L \mathcal{M}_l.
\end{equation*}
Observe that the partition changes only at cancellation times.

Fix time $\bar t$ and fix a maximal homogeneous positive (resp. negative) interval of waves $\mathcal M_l$. Let us define \emph{the effective flux function} $\freal_{\bar t} : \bigcup_{s \in \mathcal M_l} (\hat w(s) - \e, \hat w(s)] \longrightarrow \R$  (resp. $\freal_{\bar t} : \bigcup_{s \in \mathcal M_l} [\hat w(s), \hat w(s) +\e) \longrightarrow \R$) as any $C^{1,1}$ function satisfying the following condition:
\begin{equation}
\label{E_defin_freal}
\frac{d^2 \freal_{\bar t}}{dw^2}(w) := \frac{\partial^2 f}{\partial w^2}(w, v) \quad \text{for a.e. $w$,}
\end{equation}
where $v = v(\bar t, \mathtt x(\bar t,s))$ for any $s$ such that $w \in (\hat w(s) - \e, \hat w(s)]$ (resp. $w \in [\hat w(s), \hat w(s) +\e)$).

\begin{remark}
\label{R_freal}
Let us observe what follows.
\begin{enumerate}
\item \label{Pt_freal_no_dependence_on_M} To simplify the notation we do not write the explicit dependence of $\freal_{\bar t}$ on the homogeneous interval $\mathcal M_l$. No confusion should occur in the following.
\item \label{Pt_freal_defined_up_to_linear} The effective flux function $\freal_{\bar t}$ is defined up to affine functions. 
\item \label{Pt_freal_diff_sigma} Let $\mathcal I \subseteq \mathcal M_l$ be a positive (resp. negative) interval of waves at time $\bar t$. Assume that $\mathcal I \ni s \mapsto v(\bar t, \mathtt x(\bar t,s))$ is identically equal to some $\bar v$ on $\mathcal I$. Then $\freal_{\bar t}$ coincides with $f(\cdot, \bar v)$ on $\bigcup_{s \in \mathcal{I}} (\hat w(s) -\e, \hat w(s)]$ (resp. $\bigcup_{s \in \mathcal{I}} [\hat w(s),\hat w(s)+\e)$) up to affine functions. Hence, by Proposition \ref{P_conv_affine}, $s,s' \in \mathcal I$ are divided by the Riemann problem $\mathcal I$ with flux function $\freal_{\bar t}$ if and only if they are divided by the same Riemann problem with flux function $f(\cdot,\bar v)$. More precisely 
\begin{equation*}
\sigmaent (\freal_{\bar t}, \mathcal I, s') - \sigmaent (\freal_{\bar t}, \mathcal I, s) = 
\sigmaent (f(\cdot, \bar v), \mathcal I, s') - \sigmaent (f(\cdot, \bar v), \mathcal I, s).
\end{equation*}
Similarly, if $\mathcal I, \mathcal I_1, \mathcal I_2 \subseteq \mathcal M_l$ are intervals of waves at time $\bar t$ such that $\mathcal I_1, \mathcal I_2 \subseteq \mathcal I$ and $\mathcal I \ni s \mapsto v(\bar t, \mathtt x(\bar t,s))$ is identically equal to some $\bar v$, then
\begin{equation*}
\sigmarh(\freal_{\bar t}, \mathcal I_2) - \sigmarh(\freal_{\bar t}, \mathcal I_1) = \sigmarh(f(\cdot, \bar v), \mathcal I_2) - \sigmarh(f(\cdot, \bar v), \mathcal I_1).
\end{equation*}
\end{enumerate}
\end{remark}

\subsection{\texorpdfstring{The transversal interaction functional $\Qtrans$}{The transversal interaction functional}}
\label{Ss_transv_Q_glim}

In this section we define the standard Glimm transversal interaction functional $\Qtrans$ which will be frequently used in the following:
\[
\Qtrans(t) := \sum_{h=1}^H \sum_{\substack{s \in \W(t) \\ \mathtt x(t,s) < \mathtt x(t,v_h)}} |v_h||s|.
\]
Recall that $|v_h|$ is the strength of the wavefront $v_h$ and $|s| = \e$ is the strength of the wave $s$.
The following proposition is standard, see for example \cite{bre_00}.

\begin{proposition}
\label{P_Qtrans}
The following hold:
\begin{enumerate}
\item $\Qtrans (0) \leq \TV(v(0, \cdot)) \TV(w(0,\cdot))$;
\item $\Qtrans$ is positive, piecewise constant, right continuous and not increasing; moreover, at each transversal interaction times $t_j$, it holds
\[
\Qtrans (t_j) - \Qtrans (t_{j-1}) = - |v_{h(j)}| |\W(t_j, x_j)|,
\]
where $v_{h(j)}$ is the wavefront of the first family involved in the transversal interaction at time $t_j$. 
\end{enumerate}
\end{proposition}

\section{The main theorem}
\label{section_W_main_thm}

The rest of the paper is devoted to prove our main result, namely Theorem \ref{W_main_thm}. For easiness of the reader we will repeat the statement below.

As in the previous section, let $(t,x) \mapsto (w_{\e}(t,x), v_\e(t,x))$ be an $\e$-wavefront solution of the Cauchy problem \eqref{cauchy}; consider the enumeration of waves for the function $w_\e(t,\cdot)$ and the related position function $\mathtt x = \mathtt x(t,s)$ and speed function $\sigma = \sigma(t,s)$ constructed in previous section.  Fix a wave $s \in \W(0)$ and consider the function $t \mapsto \sigma(t,s)$. By construction it is finite valued until the time $T(s)$, after which its value becomes $+\infty$; moreover it is piecewise constant, right continuous, with jumps possibly located at times $t = t_j, j \in 1,\dots,J$.

The results we are going to prove is

\begin{theorem2}
The following holds:
\[
\begin{split}
\sum_{j=1}^{J} \sum_{s \in \W(t_j)} |\sigma(t_j, s)& - \sigma(t_{j-1}, s)||s| \\
\leq &~ \bigg[ 3 \|D^2_{ww}f\| +12\log(2)\|D^3_{wwv}f\| \TV(v(0,\cdot))  \bigg] \TV(w(0,\cdot))^2 \\
&~ + \|D^2_{wv} f\| \TV(w(0,\cdot)) \TV(v(0,\cdot)),
\end{split}
\]
where $|s| := \e$ is \emph{the strength of the wave $s$}.
\end{theorem2}



The first step in order to prove Theorem \ref{W_main_thm} is to reduce the quantity we want to estimate, namely
\[
\sum_{j=1}^{J} \sum_{s \in \W(t_j)} |\sigma(t_j, s) - \sigma(t_{j-1}, s)||s|,
\]
to three different quantities, which requires three separate estimates, according to $(t_j,x_j)$ being an interaction/cancellation/transversal interaction point:
\[
\begin{split}
\sum_{j=1}^{J} \sum_{s \in \W(t_j)} |\sigma(t_j, s) - \sigma(t_{j-1}, s)||s| 
=&~ \sum_{\substack{(t_j,x_j) \\ \text{interaction}}} \sum_{s \in \W(t_j)} |\sigma(t_j, s) - \sigma(t_{j-1}, s)||s| \\
&~ + \sum_{\substack{(t_j,x_j) \\ \text{cancellation}}} \sum_{s \in \W(t_j)} |\sigma(t_j, s) - \sigma(t_{j-1}, s)||s| \\
&~ + \sum_{\substack{(t_j,x_j) \\ \text{transversal} \\ \text{interaction}}} \sum_{s \in \W(t_j)} |\sigma(t_j, s) - \sigma(t_{j-1}, s)||s|.
\end{split}
\]

The estimates on transversal interaction and cancellation points are fairly easy. Let us begin with the one related to transversal interaction points.

\begin{proposition}
\label{P_trans_ch_sped}
Let $(t_j,x_j)$ be a transversal interaction point. Then
\[
\sum_{s \in \W(t_j)} |\sigma(t_j, s) - \sigma(t_{j-1}, s)||s|  \leq \lVert D^2_{wv}f \lVert_{L^\infty}  |v_{h(j)}| |\W(t_j,x_j)|,
\]
where $|v_{h(j)}|$ is the wavefront of the first family involved in the transversal interaction at time $t_j$.
\end{proposition}
\begin{proof}
Set $w_L := w_\e(t_j,x_j-)$, $w_R := w_\e(t_j, x_j)$. Assume by simplicity $w_L < w_R$, the other case being similar. Recall that $v_{h(j)}^-, v_{h(j)}^+$ are the left and right state respectively of the wavefront $v_{h(j)}$. By Proposition \ref{P_estim_diff_conv}, for any $s \in \W(t_j,x_j)$,
\[
\begin{split}
&~|\sigma(t_j, s) - \sigma(t_{j-1}, s)| \\
&~ \quad \ \ = \bigg(\frac{d}{dw} \conv_{[w_L, w_R]} f_\e \bigg) \Big( (\hat w(s) -\e, \hat w(s)),  v_{h(j)}^+ \Big) - \bigg(\frac{d}{dw} \conv_{[w_L, w_R]} f_\e \bigg)\Big( (\hat w(s)-\e, \hat w(s)),  v_{h(j)}^- \Big) \\
&~ \overset{\text{(Prop. \ref{P_estim_diff_conv})}}{\leq} \bigg\|\frac{\partial}{\partial w} f(\ \cdot \ , v_{h(j)}^+) - \frac{\partial}{\partial w} f(\ \cdot \ , v_{h(j)}^-) \bigg\|_{L^{\infty}} \\
&~ \quad \ \ \leq \|D^2_{wv} f\|_{L^{\infty}} |v_{h(j)}^+ - v_{h(j)}^-|.
\end{split}
\]
Hence, observing that the only waves which change speed are those in $\W(t_j,x_j)$,
\[
\begin{split}
\sum_{s \in \W(t_j)} |\sigma(t_j, s) - \sigma(t_{j-1}, s)||s|  &~  = \sum_{s \in \W(t_j,x_j)} |\sigma(t_j, s) - \sigma(t_{j-1}, s)||s| \\
&~ \leq \|D^2_{wv} f\|_{L^{\infty}} |v_{h(j)}^+ - v_{h(j)}^-| \sum_{s \in \W(t_j)} |s| \\
&~ = \|D^2_{wv} f\|_{L^{\infty}} |v_{h(j)}| |\W(t_j,x_j)|. \qedhere
\end{split}
\]
\end{proof}

\begin{corollary}
\label{C_trans_1}
It holds
\[
\sum_{\substack{(t_j,x_j) \\ \rm{transversal} \\ \rm{interaction}}} \sum_{s \in \W(t_j)} |\sigma(t_j, s) - \sigma(t_{j-1}, s)||s| \leq \|D^2_{wv} f\|_{L^{\infty}} \TV(w(0,\cdot)) \TV(v(0,\cdot)). 
\]
\end{corollary}
\begin{proof}
The proof is an easy consequence of Proposition \ref{P_Qtrans} and the fact that, by previous proposition, for any transversal interaction time $t_j$,
\[
\begin{split}
\sum_{s \in \W(t_j)} |\sigma(t_j, s) - \sigma(t_{j-1}, s)||s| &~\leq \|D^2_{wv} f\|_{L^{\infty}} |v_{h(j)}| |\W(t_j,x_j)| \\
&~ = \|D^2_{wv} f\|_{L^{\infty}} \Big( \Qtrans(t_{j-1}) - \Qtrans(t_j) \Big). \qedhere
\end{split}
\]
\end{proof}

Let us now prove the estimate related to the cancellation points. First of all define for each cancellation point $(t_j,x_j)$ the \emph{amount of cancellation} as follows:
\begin{equation}
\label{W_canc_0}
\mathcal{C}(t_j,x_j) := \TV(w_\e(t_{j-1}, \cdot)) - \TV(w_\e(t_j,\cdot)).
\end{equation}

\begin{proposition}
\label{W_canc_3}
Let $(t_j,x_j)$ be a cancellation point. Then
\[
\sum_{s \in \W(t_j)} |\sigma(t_j, s) - \sigma(t_{j-1}, s)||s|  \leq \lVert D^2_{ww}f \lVert_{L^\infty}  \TV(w(0,\cdot)) \C(t_j,x_j).
\]
\end{proposition}
\begin{proof}
Let $w_L, w_M$ be respectively the left and the right state of the left wavefront involved in the collision at point $(t_j,x_j)$ and let $w_M, w_R$ be respectively the left and the right state of the right wavefront involved in the collision at point $(t_j,x_j)$, so that $w_L = {\displaystyle \lim_{x \nearrow x_j}} w_\e(t_j, x)$ and $w_R = w_\e(t_j, x_j)$. Without loss of generality, assume $w_L < w_R < w_M$. Finally set $\bar v:= v(t_j,x_j)$.

Then we have
\begin{equation}
\label{W_canc_1}
\begin{split}
& \sum_{s \in \W(t_j)} |\sigma(t_j, s) - \sigma(t_{j-1}, s)||s|  \\ 
& \quad \ \ = \sum_{s \in \W(t_j)}  \bigg[ \bigg(\frac{d}{dw} \conv_{[w_L, w_R]} f_\e \bigg) \Big( (\hat w(s) -\e, \hat w(s)), \bar v \Big) - \bigg(\frac{d}{dw} \conv_{[w_L, w_M]} f_\e \bigg)\Big( (\hat w(s)-\e, \hat w(s)), \bar v \Big) \bigg]|s| \\
& \overset{\text{(Prop. \ref{differenza_vel})}}{\leq} 
\sum_{s \in \W(t_j)}  \bigg[ \bigg(\frac{d}{dw} \conv_{[w_L, w_R]} f_\e \bigg)(w_R-, \bar v) - \bigg(\frac{d}{dw} \conv_{[w_L, w_M]} f_\e \bigg)(w_R-, \bar v) \bigg]|s| \\
& \quad \ \ =  \bigg[ \bigg(\frac{d}{dw} \conv_{[w_L, w_R]} f_\e \bigg)(w_R-, \bar v) - \bigg(\frac{d}{dw} \conv_{[w_L, w_M]} f_\e \bigg)(w_R-, \bar v) \bigg]  \sum_{s \in \W(t_j)} |s| \\
& \quad \ \ \leq \bigg[ \bigg(\frac{d}{dw} \conv_{[w_L, w_R]} f_\e \bigg)(w_R-, \bar v) - \bigg(\frac{d}{dw} \conv_{[w_L, w_M]} f_\e \bigg)(w_R-, \bar v) \bigg]  \TV(w(0,\cdot)) .
\end{split}
\end{equation}
Now observe that, by Proposition \ref{diff_vel_proporzionale_canc}, 
\begin{equation}
\label{W_canc_2}
\begin{split}
\bigg(\frac{d}{dw} \conv_{[w_L, w_R]} f_\e \bigg)(w_R-, \bar v) - \bigg(\frac{d}{dw} \conv_{[w_L, w_M]} f_\e \bigg)(w_R-, \bar v) 
\leq&~ \Big\| \frac{d^2f(\cdot, \bar v)}{dw^2} \Big\|_{L^\infty} (w_M - w_R) \\
\leq&~ \lVert D^2_{ww}f \lVert_{L^\infty} (w_M - w_R) \\
\leq&~ \lVert D^2_{ww}f \lVert_{L^\infty} \mathcal{C}(t_j,x_j).
\end{split}
\end{equation}

Hence, from \eqref{W_canc_1} and \eqref{W_canc_2}, we obtain
\begin{equation*}
\sum_{s \in \W(t_j)} |\sigma(t_j, s) - \sigma(t_{j-1}, s)||s|  \leq   
\lVert D^2_{ww}f \lVert_{L^\infty} \TV(w(0,\cdot)) \mathcal{C}(t_j,x_j). \qedhere
\end{equation*}
\end{proof}

\begin{corollary}
\label{W_canc_4}
It holds
\[
\sum_{\substack{(t_j,x_j) \\ \rm{cancellation}}} \sum_{s \in \W(t_j)} |\sigma(t_j, s) - \sigma(t_{j-1}, s)||s| \leq \lVert D^2_{ww}f \lVert_{L^\infty} \TV(w(0,\cdot))^2.
\]
\end{corollary}
\begin{proof}
From \eqref{bd_su_dato_iniziale}, \eqref{W_canc_0} and Proposition \ref{W_canc_3} we obtain
\[
\begin{split}
\sum_{\substack{(t_j,x_j) \\ \text{cancellation}}} & \sum_{s \in \W(t_j)} |\sigma(t_j, s) - \sigma(t_{j-1}, s)||s|\\
& \leq  \lVert D^2_{ww}f \lVert_{L^\infty} \TV(w(0,\cdot)) \sum_{j=1}^J \Big[ \TV(w_\e(t_{j-1}, \cdot)) - \TV(w_\e(t_j, \cdot)) \Big] \\
& \leq  \lVert D^2_{ww}f \lVert_{L^\infty} \TV(w(0,\cdot)) \Big[ \TV(w_\e(0,\cdot)) - \TV(w_\e(t_J, \cdot)) \Big] \\
& \leq  \lVert D^2_{ww}f \lVert_{L^\infty} \TV(w(0,\cdot))^2,
\end{split}
\]
thus concluding the proof of the corollary.
\end{proof}

From now on, our aim is to prove that 
\begin{equation*}
\begin{split}
\sum_{\substack{(t_j,x_j) \\ \text{interaction}}} \sum_{s \in \W(t_j)} |\sigma(t_j, s) &- \sigma(t_{j-1}, s)||s| \\
&~ \leq 
2 \Big[\|D^2_{ww} f\| + 6\log(2) \|D^3_{wwv} f\| \TV(v(0,\cdot)) \Big] \TV(w(0,\cdot))^2.
\end{split}
\end{equation*}

As outlined in Section \ref{Sss_sketch_proof}, the idea is the following: we  define a positive valued functional $\fQ = \fQ(t)$, $t \geq 0$, such that $\fQ$ is piecewise constant in time, right continuous, with jumps possibly located at times $t_j, j=1,\dots,J$ and such that 
\begin{equation}
\label{boundQzero}
\fQ(0) \leq \lVert D^2_{ww}f \lVert_{L^\infty} \TV(w(0,\cdot))^2.
\end{equation}
Such a functional will have three properties:
\begin{enumerate}
\item for each $j$ such that $(t_j,x_j)$ is an interaction point, $\fQ$ is decreasing at time $t_j$ and its decrease bounds the quantity we want to estimate at time $t_j$ as follows:
\begin{equation}
\label{W_decrease}
\sum_{s \in \W(t_j)} |\sigma(t_j, s) - \sigma(t_{j-1}, s)||s| \leq 2 \Big[ \fQ(t_{j-1}) - \fQ(t_j) \Big];
\end{equation}
this is proved in Corollary \ref{W_decreasing};
\item for each $j$ such that $(t_j, x_j)$ is a cancellation point, $\fQ$ is decreasing; this will be an immediate consequence of the definition of $\fQ$;
\item for each $j$ such that $(t_j,x_j)$ is a transversal interaction point, $\fQ$ can increase at most by 
\begin{equation}
\label{W_increase}
\begin{split}
\fQ(t_j) - \fQ(t_{j-1}) &~ \leq 6 \log(2) \|D^3_{wwv}f\|_{L^\infty} |v_{h(j)}| |\W(t_j,x_j)| \TV(w(0,\cdot));
\end{split}
\end{equation}
this is proved in Theorem \ref{W_increasing}.
\end{enumerate}
Using the two estimates above, we obtain the following proposition, which completes the proof of Theorem \ref{W_main_thm}.

\begin{proposition}
\label{W_thm_interaction}
It holds
\begin{equation*}
\begin{split}
\sum_{\substack{(t_j,x_j) \\ \rm{interaction}}} \sum_{s \in \W(t_j)} |\sigma(t_j, s) &- \sigma(t_{j-1}, s)||s| \\
&~ \leq 
2 \Big[\|D^2_{ww} f\| + 6\log(2) \|D^3_{wwv} f\| \TV(v(0,\cdot)) \Big] \TV(w(0,\cdot))^2.
\end{split}
\end{equation*}
\end{proposition}

\begin{proof}
By direct computation,
\[
\begin{split}
\sum_{\substack{(t_j,x_j) \\ \text{interaction}}} \sum_{s \in \W(t_j)}& |\sigma(t_j, s) - \sigma(t_{j-1}, s)||s| \\
\text{(by \eqref{W_decrease})}  \leq &~ 2 \sum_{\substack{(t_j,x_j) \\ \text{interaction}}}  \big[ \fQ(t_{j-1}) - \fQ(t_j) \big] \\
\leq &~ 2 \Bigg[ \sum_{\substack{(t_j,x_j) \\ \text{interaction}}}  \big[ \fQ(t_{j-1}) - \fQ(t_j) \big] \\
&~ \quad \quad + \sum_{\substack{(t_j,x_j) \\ \text{cancellation}}}  \big[ \fQ(t_{j-1}) - \fQ(t_j) \big] - \sum_{\substack{(t_j,x_j) \\ \text{cancellation}}}  \big[ \fQ(t_{j-1}) - \fQ(t_j) \big]  \Bigg] \\
&~ \quad \quad + \sum_{\substack{(t_j,x_j) \\ \text{transversal}}}  \big[ \fQ(t_{j-1}) - \fQ(t_j) \big] - \sum_{\substack{(t_j,x_j) \\ \text{transversal}}}  \big[ \fQ(t_{j-1}) - \fQ(t_j) \big]  \Bigg] \\
\text{($\fQ$ decreases} & \text{ at cancellations)} \\
\leq &~ 2 \Bigg[ \sum_{j=1}^J \big[ \fQ(t_{j-1}) - \fQ(t_j) \big] 
+ \sum_{\substack{(t_j,x_j) \\ \text{transversal}}}  \big[ \fQ(t_j) - \fQ(t_{j-1}) \big]   \Bigg] \\
\text{(by \eqref{W_increase})}  \leq &~  2 \Bigg[ \sum_{j=1}^J \big[ \fQ(t_{j-1}) - \fQ(t_j) \big] \\
&~ \quad \quad + 6 \log(2) \|D^3_{wwv}f\|_{L^\infty} \TV(w(0,\cdot)) 
\sum_{\substack{(t_j,x_j) \\ \text{transversal}}} |v_{h(j)}| |\W(t_j,x_j)| \Bigg] \\
= &~  2 \Bigg[ \sum_{j=1}^J \big[ \fQ(t_{j-1}) - \fQ(t_j) \big] \\
&~ \quad \quad + 6 \log(2) \|D^3_{wwv}f\|_{L^\infty} \TV(w(0,\cdot)) 
\sum_{\substack{(t_j,x_j) \\ \text{transversal}}} \Big(\Qtrans(t_{j-1}) - \Qtrans(t_j)\Big) \Bigg] \\
\text{(by Proposit} & \text{ion \ref{P_Qtrans})} \\
\leq &~ 2 \Big[ \fQ(0) + 6 \log(2) \|D^3_{wwv}f\|_{L^\infty} \TV(w(0,\cdot)) \Qtrans(0) \Big] \\
\text{(by \eqref{boundQzero} an}& \text{d Proposition \ref{P_Qtrans})} \\
\leq &~ 2 \Big[\|D^2_{ww} f\| + 6\log(2) \|D^3_{wwv} f\| \TV(v(0,\cdot)) \Big] \TV(w(0,\cdot))^2,
\end{split}
\]
thus concluding the proof of the proposition.
\end{proof}

In the remaining part of the paper we prove estimates \eqref{W_decrease} and \eqref{W_increase}.

\subsection{Analysis of waves collisions}
\label{W_waves_collision}

In this section we define the notion of pairs of waves which \emph{have never interacted before a fixed time $t$} and pairs of waves which \emph{have already interacted} and, for any pair of waves which have already interacted, we associate an interval of waves and a partition of this interval, which in some sense summarize their past common history.

\begin{definition}
\label{interagite_non_interagite}
Let $\bar t$ be a fixed time and let $s,s' \in \W(\bar t)$. We say that \emph{$s,s'$ interact at time $\bar t$} if $\mathtt x(\bar t, s) = \mathtt x(\bar t, s')$.  

We also say that \emph{they have already interacted at time $\bar t$} if there is $t \leq \bar t$ such that $s,s'$ interact at time $t$. Moreover we say that \emph{they have never interacted at time $\bar t$} if for any $t \leq \bar t$, they do not interact at time $t$. 
\end{definition}

\begin{lemma}
\label{W_interagite_stesso_segno}
Assume that the waves $s, s'$ have already interacted at time $\bar t$. Then they have the same sign.
\end{lemma}

\begin{proof}
Easy consequence of definition of enumeration of waves and the fact that $\mathcal S(s)$ is independent of $t$. 
\end{proof}

\begin{lemma}
\label{W_quelle_in_mezzo_hanno_int}
Let $\bar t$ be a fixed time, $s,s' \in \W(\bar t)$, $s < s'$. Assume that $s, s'$ have already interacted at time $\bar t$. If $p, p' \in \W(\bar t)$ and $s \leq p \leq p' \leq s'$, then $p, p'$ have already interacted at time $\bar t$.
\end{lemma}
\begin{proof}
Let $t$ be the time such that $s,s'$ interact at time $t$. Clearly $s,s',p,p' \in \W(t) \supseteq \W(\bar t)$. Since for $t$ fixed, $\mathtt x$ is increasing on $\W(t)$, it holds $\mathtt x(t,s) = \mathtt x(t,p) = \mathtt x(t, p') = \mathtt x(t, s')$. 
\end{proof}

%

\begin{definition}
\label{W_waves_divided}
Let $s,s' \in \W(\bar t)$ be two waves which have already interacted at time $\bar t$. We say that \emph{$s,s'$ are divided in the real solution at time $\bar t$} if 
\[
(\mathtt x(\bar t, s), \sigma(\bar t, s)) \neq (\mathtt x(\bar t, s'), \sigma(\bar t, s')),
\]
i.e. if at time $\bar t$ they have either different position, or the same position but different speed.

\noindent If they are not divided in the real solution, we say that \emph{they are joined in the real solution}.
\end{definition}

\begin{remark}
\label{rem_divise_solo_in_cancellazioni}
It $\bar t \neq t_j$ for each $j$, then two waves are divided in the real solution if and only if they have different position. The requirement to have different speed is needed only at cancellation and transversal interaction times.
\end{remark}

Fix a time $\bar t$ and two waves $s < s'$ which have already interacted at time $\bar t$ and assume that $s,s'$ are divided in the real solution at time $\bar t$. Define the \emph{time of last interaction} $\mathtt T(\bar t,s,s')$ by the formula
\begin{equation}
\label{E_ttT_def}
\mathtt T(\bar t,s,s') := \max\{t \leq \bar t \ | \ \mathtt x(t,s) = \mathtt x(t,s')\}.
\end{equation}
Moreover set
\[
\mathtt X(\bar t,s,s') := \mathtt x(\mathtt T(\bar t,s,s'), s) = \mathtt x(\mathtt T(\bar t,s,s'),s').
\]
Finally define
\begin{equation}
\label{E_calI_def_1}
\mathcal I(\bar t,s,s') := \W\Big(\mathtt T(\bar t,s,s'), \mathtt X(\bar t,s,s')\Big) \cap \W(\bar t).
\end{equation}
It is easy to see that $\mathcal I(\bar t,s,s')$ is an interval of waves at time $\bar t$ (i.e. with respect to the function $w_\e (\bar t, \cdot)$ and the related enumeration of waves). Observe also that it changes only at interaction/cancellation/transversal interaction times. It is immediate to see that if $\mathtt x(\bar t,s) = \mathtt x(\bar t,s')$, but $s,s'$ are divided in the real solution at time $\bar t$ (i.e. $\sigma(\bar t,s) < \sigma(\bar t, s')$), then $\bar t = \mathtt T(\bar t,s,s')$ and 
\[
\mathcal I(\bar t,s,s') = \W(\bar t,\mathtt x(t,s)) = \W(\bar t,\mathtt x(t,s')).
\]
The interesting case we are interested in is for $\bar t > \mathtt T(\bar t,s,s')$.

Let us now define a partition $\mathcal P(\bar t,s,s')$ of the interval of waves $\mathcal I(\bar t,s,s')$ by recursion on $\bar t = t_0, \dots, t_J$, $\bar t \geq \mathtt T(\bar t,s,s')$, $s,s'$ divided in the real solution at time $\bar t$, as follows.

For $\bar t = \mathtt T(\bar t,s,s') = t_{\bar j}$, for some $\bar j \in \{0,\dots,J\}$, then $\mathcal P(t_{\bar j},s,s')$ is given by the equivalence relation
\begin{eqnarray*}
p \sim p'  & \Longleftrightarrow & p,p' \text{ are not divided in the real solution at time $t_{\bar j}$ or, equivalently,} \\
& & \text{they are not divided by the Riemann problem $\W(t_{\bar j}, \mathtt x(t_{\bar j},s))$ with flux function $\freal_{t_{\bar j}}$}.
\end{eqnarray*}
On the other hand, if $\bar t = t_{\bar j} > \mathtt T(\bar t,s,s')$ for some $\bar j \in \{1,\dots,J\}$ (i.e. $s,s'$ are divided in the real solution also at time $t_{\bar{j}-1}$), then $\mathcal P(t_{\bar{j}},s,s')$ is given by the equivalence relation
\begin{eqnarray*}
p \sim p'  & \Longleftrightarrow &
p,p' \text{ belong to the same equivalence class $\mathcal J \in\mathcal P(t_{\bar{j}-1},s,s')$ at time $t_{\bar{j}-1}$} \\
& & \text{and the Riemann problem $\mathcal J \cap \W(t_{\bar{j}})$ with flux function $\freal_{t_{\bar{j}}}$ does not divide them}.
\end{eqnarray*}
As a consequence of Remark \ref{R_partition_iow} and the fact that both $\W(t_{\bar j}, \mathtt x(t_{\bar j},s))$ and $\mathcal J \cap \W(t_{\bar j})$ are interval of waves at time $t_{\bar j}$, we immediately see that each element of the partition $\mathcal P(t_{\bar j},s,s')$ is an entropic interval of waves (w.r.t. flux function $\freal_{t_{\bar j}}$) and the relation $<$ introduced in Definition \ref{D_order} is a total order on $\mathcal P(t_{\bar j},s,s')$.

\begin{proposition}
\label{P_divise_partizione_implica_divise_realta}
For any $j = 0, \dots, J$ such that $s,s'$ are divided at time $t_j$ in the real solution,
if $r,r' \in \mathcal I(t_j,s,s')$ are not divided by the partition $\mathcal P(t_j,s,s')$, then they are not divided in the real solution at time $t_j$.
\end{proposition}

\begin{proof}
We prove the proposition by induction. Clearly we have only to consider the the case $t_j > \mathtt T(t_j,s,s')$, since the case $t_j = \mathtt T(t_j,s,s')$ is immediate.

Assume thus $s,s'$ already divided at time $t_{j-1}$. Take $r,r' \in \mathcal I(t_j,s,s')$, not divided by the partition $\mathcal P(t_j,s,s')$. By definition, this means that $r,r'$ belong to the same equivalence class $\mathcal J \in \mathcal P(t_{j-1},s,s')$ at time $t_{j-1}$ and the Riemann problem  $\mathcal J \cap \W(t_j)$ with flux function $\freal_{t_j}$ does not divide them. Assume by contradiction that they are divided in the real solution at time $t_j$. This means that $\mathtt x(t_j,r) = \mathtt x(t_j,r') = x_j$ and the Riemann problem $\W(t_j,x_j)$ with flux function $\freal_{t_j}$ (see Remark \ref{R_freal}, Point \eqref{Pt_freal_diff_sigma}) divides $r,r'$. Since by inductive assumption waves in $\mathcal J$ are not divided in the real solution at time $t_{j-1}$ and $r,r' \in \mathcal J \cap \W(t_j,x_j)$, then $\mathcal J \cap \W(t_j) \subseteq \W(t_j, x_j)$. By Corollary \ref{stesso_shock}, this is a contradiction.
\end{proof}

\begin{definition}
Let $A,B$ two sets, $A \subseteq B$. Let $\mathcal P$ be a partition of $B$. We say that $\mathcal P$  \emph{can be restricted to $A$} if for any $C \in \mathcal P$, either $C \subseteq A$ or $C \subseteq B \setminus A$. We also write
\[
\mathcal P|_{A} := \Big\{C \in \mathcal P \ \Big| \ C \subseteq A \Big\}.
\]
\end{definition}
Clearly $\mathcal P$ can be restricted to $A$ if and only if it can be restricted to $B \setminus A$.

\begin{proposition}
\label{P_partition_restr}
Let $j = 0, \dots, J$ fixed. Let $s,s',p,p' \in \W(t_j)$, $p \leq s < s' \leq p'$; assume that $p,p'$ have already interacted at time $t_j$ and $s,s'$ are divided in the real solution at time $t_j$. Then $\mathcal P(t_j, p,p')$ can be restricted both to $\mathcal I (t_j,s,s') \cap \mathcal I(t_j,p,p')$ and to $\mathcal I(t_j,p,p') \setminus \mathcal I(t_j,s,s')$.

Moreover if $p,p' \in \mathcal I(t_j,s,s')$, then $\mathcal I(t_j,p,p') = \mathcal I(t_j,s,s')$ and $\mathcal P(t_j,p,p') = \mathcal P(t_j,s,s')$.
\end{proposition}
\begin{proof}
Let us prove the first part of the proposition. Assume first that either $j=0$ or $s,s'$ are joined in the real solution at time $t_{j-1}$. Let $\mathcal J \in \mathcal P(t_j,p,p')$ such that $\mathcal J \cap \mathcal I(t_j,s,s') \neq \emptyset$. Since $\mathcal I(t_j,s,s') = \W(t_j, \mathtt x(t_j,s))$, by Proposition \ref{P_divise_partizione_implica_divise_realta} applied to $\mathcal I(t_j,p,p')$ and waves in $\mathcal J$, it must hold $\mathcal J \subseteq \W(t_j,x_j)$. 

\noindent Now assume that $s,s'$ are divided in the real solution at time $t_{j-1}$. Take $\mathcal J \in \mathcal P(t_j,p,p')$ and assume that $\mathcal J \cap \mathcal I(t_j,s,s') \neq \emptyset$. By definition of the equivalence classes, there is $\mathcal K \in \mathcal P(t_{j-1},p,p')$ such that $\mathcal J \subseteq \mathcal K$. Clearly $\mathcal K \cap \mathcal I(t_{j-1},s,s') \neq \emptyset$, and so, by inductive assumption, $\mathcal K \subseteq \mathcal I(t_{j-1},s,s')$. Hence
\[
\mathcal J \subseteq \mathcal K \cap \W(t_j) \subseteq
\mathcal I(t_{j-1},s,s') \cap \W(t_j) = \mathcal I(t_j,s,s'). 
\]

Let us now prove the second part of the proposition by recursion. If either $j=0$ or at time $t_{j-1}$, waves $s,s'$ are joined in the real solution, then the conclusion is obvious. Assume now that at time $t_{j-1}$ waves $s,s'$ are  divided in the real solution and assume that $\mathcal I(t_{j-1},p,p') = \mathcal I(t_{j-1},s,s')$.
Thus
\[
\mathcal I(t_j,p,p') = \mathcal I(t_{j-1},p,p') \cap \W(t_j) = \mathcal I(t_{j-1},s,s') \cap \W(t_j) = \mathcal I(t_j,s,s').
\]

\noindent
Finally assume that $r,r' \in \mathcal I(t_j,s,s') = \mathcal I(t_j,p,p')$. Then it holds
\begin{eqnarray*}
r \sim r' \text{ w.r.t. the partition $\mathcal P(t_j,p,p')$} & \Longleftrightarrow &
r,r' \text{ belong to the same equivalence} \\
& & \text{class $\mathcal J \in\mathcal P(t_{j-1},p,p')$ at time $t_{j-1}$} \\
& & \text{and the Riemann problem $\mathcal J \cap \W(t_j)$} \\
& &  \text{with flux function $\freal_{t_j}$ does not divide them} \\
& \Longleftrightarrow &
r,r' \text{ belong to the same equivalence} \\
& & \text{class $\mathcal J \in\mathcal P(t_{j-1},s,s')$ at time $t_{j-1}$} \\
& & \text{and the Riemann problem $\mathcal J \cap \W(t_j)$} \\
& &  \text{with flux function $\freal_{t_j}$ does not divide them} \\
& \Longleftrightarrow & r \sim r' \text{ w.r.t. the partition $\mathcal P(t_j,s,s')$}.
\end{eqnarray*}
Hence $\mathcal P(t_j,p,p') = \mathcal P(t_j,s,s')$.
\end{proof}

\subsection{\texorpdfstring{Definition of $\fQ$}{Definition of Q}}
\label{Sss_fQ_def}

We can finally define the functional $\fQ$ and prove that it satisfies inequalities \eqref{W_decrease} and \eqref{W_increase}.

Let $s<s'$ be two waves. Let $t_j$ be a transversal interaction time;  assume that $s,s'$ are divided in the real solution both at $t_{j-1}$ and at $t_j$, and have already interacted. For any $p,p' \in \mathcal I(t_{j-1},s,s') = \mathcal I(t_j,s,s')$, let $\mathcal J_p, \mathcal J_{p'}$ be the element of $\mathcal P(t_{j-1},s,s')$ containing $p,p'$ respectively. Define 
\[
M(t_j,s,s')[p,p'] := \sum_{\substack{\mathcal J \in \mathcal P(t_{j-1},s,s') \\ \mathcal J_p \leq \mathcal J \leq \mathcal J_{p'} \\ \mathcal J \subseteq \W(t_j,x_j)}} |\mathcal J | .
\]
The above number is the length of the minimal interval containing $p$, $p'$, obtained by union of components of $\mathcal P(t_{j-1},s,s')$ which are subsets of $\W(t_j,x_j)$.

Let $s<s'$ be two waves. For any time $t_j < \min\{T(s), T(s')\}$ such that $s,s'$ are divided at time $t_j$ in the real solution and have already interacted, define by recursion the map 
\[
\pi(t_j,s,s'): \Big( \mathcal I(t_j,s,s') \times \mathcal I(t_j,s,s') \Big)\cap \{ p < p'\} \longrightarrow \R
\]
as follows:
\begin{enumerate}
\item if either $j=0$ or $s,s'$ are joined in the real solution at time $t_{j-1}$, set $\pi(t_j,s,s')[p,p'] = 0$ for any $(p,p') \in  \Big( \mathcal I(t_j,s,s') \times \mathcal I(t_j,s,s') \Big)\cap \{ p < p'\}$;
\item if $s,s'$ are divided in the real solution at time $t_{j-1}$,
    \begin{enumerate}
    \item if $t_j$ is an interaction or a cancellation point, set
$\pi(t_j,s,s')[p,p'] := \pi(t_{j-1},s,s')[p,p']$ ;    
    \item if $t_j$ is a transversal interaction point, set
    \begin{equation}
    \label{E_max_ddiff_trans}
    \pi(t_j,s,s')[p,p'] := \pi(t_{j-1},s,s')[p,p'] + 2\|D^3_{wwv}f\|_{L^\infty} |v_{h(j)}| M(t_j,s,s')[p,p'].
    \end{equation}
    \end{enumerate}
\end{enumerate}

Now for any time $t_j$ and for any pair of wave $(s,s')$, $s<s'$, define \emph{the weight $\mathfrak q(t_j,s,s')$ of the pair of waves $s$, $s'$ at time $t_j$} in the following way:
\begin{equation}
\label{W_mathfrak_q}
\mathfrak{q}(t_j, s, s') :=
\begin{cases}
0 & \text{$s,s'$ joined at time $t_j$ in the real solution}\\
\dfrac{\pi(t_j,s,s')[s,s']}{|\hat w(s') - (\hat w(s)-\mathcal{S}(s) \e)|} & 
\text{$s,s'$ divided at time $t_j$ and already interacted}, \\
 \|D^2_{ww} f\|_{L^\infty}& \text{$s,s'$ never interacted.}
\end{cases}
\end{equation}


Finally set
\begin{equation}
\label{E_frak_Q_def}
\mathfrak{Q}(t_j) :=  \sum_{\substack{s,s' \in \W(t_j) \\ s < s'}} \mathfrak{q}(t_j, s, s') |s||s'|,
\end{equation}
and $\fQ(t) = \fQ(t_j)$ for $t \in [t_j, t_{j+1})$ (or $[t_J, +\infty)$). Recall that $|s| = |s'| = \e$ is the strength of the waves $s,s'$ respectively.

It is immediate to see that $\fQ$ is positive, piecewise constant, right continuous, with jumps possibly located at times $t_j, j=1,\dots,J$, and $\fQ(0) \leq \|D^2_{ww} f\|_{L^\infty} \TV(w(0,\cdot))^2$. In the next two sections we prove that it also satisfies inequality \eqref{W_decrease} and \eqref{W_increase}. This completes the proof of Proposition \ref{W_thm_interaction}.

Indeed, the fact that $\mathfrak{Q}$ decreases at cancellation is simply due to the fact that the weights are not increasing, and some terms of the sum \eqref{E_frak_Q_def} are canceled.

\subsection{\texorpdfstring{Decreasing part of $\fQ$}{Decreasing part of Q}}
\label{Ss_Q_decrease}

This section is devoted to prove inequality \eqref{W_decrease}. We will prove it only in the case of a positive interaction point, the negative case being completely similar.

\begin{lemma}
\label{L_estim_interval_of_partition}
Let $t_j$ be a fixed time. Let $s<s'$ be two waves, divided in the real solution at time $t_j$, but which have already interacted. Assume $s,s'$ positive. Let $\mathcal J, \mathcal J' \in \mathcal P(t_j,s,s')$, $\mathcal J < \mathcal J'$, and $p \in \mathcal J$, $p' \in \mathcal J'$. It holds
\begin{equation}
\label{E_estim_interval_of_partition}
\sigmarh(\freal_{t_j}, \mathcal J) - \sigmarh(\freal_{t_j}, \mathcal J')  \leq  \pi(t_j,s,s')[p,p'].
\end{equation}
\end{lemma}
\begin{proof}
The proof is by induction on times $(t_j)_{j=0, \dots, J}$. Notice that the r.h.s. of \eqref{E_estim_interval_of_partition} is greater or equal than $0$. If either $j=0$ or at time $t_{j-1}$ waves $s,s'$ are joined in the real solution, the l.h.s. of \eqref{E_estim_interval_of_partition} is negative, hence \eqref{E_estim_interval_of_partition} holds. 

Assume now that at time $t_{j-1}$ waves $s,s'$ are divided in the real solution.

\smallskip
{\it Case 1.} Assume $t_j$ is an interaction point. In this case $\mathcal I(t_j,s,s') = \mathcal I(t_{j-1},s,s')$, $\mathcal P(t_j,s,s') = \mathcal P(t_{j-1},s,s')$, $\pi(t_j,s,s') = \pi(t_{j-1},s,s')$, $\freal_{t_j} = \freal_{t_{j-1}}$; hence by inductive assumption we are done.

\smallskip
{\it Case 2.} Assume that $t_j$ is a cancellation point and w.l.o.g. suppose that the cancellation is on the right of $\mathcal I(t_{j-1},s,s')$. It is not difficult to see that there is at most one interval $\mathcal K \in \mathcal P(t_{j-1},s,s')$ which is reduced (but not completely canceled) and possibly split by the cancellation. 
\begin{itemize}
\item If $\mathcal J, \mathcal J' \subseteq \mathcal K$, the l.h.s. of \eqref{E_estim_interval_of_partition} is negative and we are done. 

\item If $\mathcal J < \mathcal K$ and $\mathcal J' \subseteq \mathcal K$, $\sigmarh(\freal_{t_{j-1}}, \mathcal J) = \sigmarh(\freal_{t_j}, \mathcal J)$, while, by Proposition \ref{vel_aumenta}, $\sigmarh(\freal_{t_{j-1}}, \mathcal K) \leq \sigmarh(\freal_{t_j}, \mathcal J')$. Hence\begin{equation*}
\sigmarh(\freal_{t_{j}}, \mathcal J) - \sigmarh(\freal_{t_{j}}, \mathcal J') \leq \sigmarh(\freal_{t_{j-1}}, \mathcal J) - \sigmarh(\freal_{t_{j-1}}, \mathcal K) \leq \pi(t_{j-1},s,s')[p,p'] = \pi(t_j,s,s')[p,p'],
\end{equation*}
where the second inequality comes from the inductive assumption at time $t_{j-1}$. 

\item If $\mathcal J,\mathcal J' < \mathcal K$, then, as before,
\begin{equation*}
\sigmarh(\freal_{t_{j}}, \mathcal J) - \sigmarh(\freal_{t_{j}}, \mathcal J') = \sigmarh(\freal_{t_{j-1}}, \mathcal J) - \sigmarh(\freal_{t_{j-1}}, \mathcal J') \leq \pi(t_{j-1},s,s')[p,p'] = \pi(t_j,s,s')[p,p'].
\end{equation*}
\end{itemize}

\smallskip
{\it Case 3.} Assume that $t_j$ is a transversal interaction point. There exist $\mathcal K, \mathcal K' \in \mathcal P(t_{j-1},s,s')$ containing $\mathcal J, \mathcal J'$ respectively. If $\mathcal J, \mathcal J' < \W(t_j,x_j)$ (or $\mathcal J, \mathcal J' > \W(t_j,x_j)$), then $\mathcal J = \mathcal K$, $\mathcal J' = \mathcal K'$ and we can use inductive assumption to conclude:
\[
\sigmarh(\freal_{t_{j}}, \mathcal J) - \sigmarh(\freal_{t_{j}}, \mathcal J') = \sigmarh(\freal_{t_{j-1}}, \mathcal J) - \sigmarh(\freal_{t_{j-1}}, \mathcal J') \leq \pi(t_{j-1},s,s')[p,p'] = \pi(t_j,s,s')[p,p'].
\]

Assume thus 
\[
\Big( \mathcal J < \W(t_j,x_j) \text{ or } \mathcal J \subseteq \W(t_j,x_j)\Big) \text{ and } \Big(\mathcal J' \subseteq \W(t_j,x_j) \text{ or } \mathcal J' > \W(t_j,x_j) \Big).
\]
We can also assume $\mathcal K < \mathcal K'$, otherwise the l.h.s. of \eqref{E_estim_interval_of_partition} is less or equal than $0$. Set $J := \bigcup_{r \in \mathcal J} (\hat w(r) -\e, \hat w(r)]$ and similarly define $J', K, K', W \subseteq \R$ as the union of the waves segments in $\mathcal J', \mathcal K, \mathcal K', \W(t_j,x_j)$, respectively. 
Let $a$ be any point such that $\sup K \leq a \leq \inf K'$. Since $\freal$ is defined up to affine function on each maximal monotone interval of waves, we can choose $\freal_{t_j}, \freal_{t_{j-1}}$ such that, $\frac{d }{dw}\freal_{t_{j-1}}(a) = \frac{d }{dw}\freal_{t_j}(a) = 0$. 

\noindent For any $w \in K$, it holds  
\begin{equation*}
\begin{split}
\bigg | \frac{d}{dw} \freal_{t_j}(w) - \frac{d}{dw} \freal_{t_{j-1}}(w) \bigg|
= &~ \bigg| \int_w^a \frac{d^2}{dw^2} \freal_{t_j}(\tau) d\tau - 
\int_w^a \frac{d^2}{dw^2} \freal_{t_{j-1}}(\tau) d\tau \bigg| \\
\leq &~  \int_{(w,a) \cap W} \bigg| \frac{\partial^2}{\partial w^2}f(\tau, v(t_j,x_j)) - \frac{\partial^2}{\partial w^2}f(\tau, v(t_j,x_j-)) \bigg| d\tau \\
\leq &~ \int_{(w,a) \cap W} \int_{v_{h(j)}^-}^{v_{h(j)}^+} \bigg| \frac{\partial^3 f}{\partial w^2 \partial v}(\tau, v) \bigg| dv d\tau \\
\leq &~ \big \| D^3_{wwv}     f \big \|_{L^\infty} \big|v_{h(j)}\big| \big|(w,a) \cap W\big| \\
\leq &~ \big \| D^3_{wwv} f \big \|_{L^\infty} \big|v_{h(j)}\big|M(t_j,s,s')[p,p'],
\end{split}
\end{equation*}
and thus
\begin{equation}
\label{E_diff_freal_peso}
\bigg \| \frac{d}{dw} \freal_{t_j} - \frac{d}{dw} \freal_{t_{j-1}} \bigg\|_{L^\infty(K)} \leq \big \| D^3_{wwv} f \big\|_{L^\infty} \big|v_{h(j)}\big|M(t_j,s,s')[p,p'].
\end{equation}
Now observe that for any $w \in J$,
\begin{equation}
\label{E_diff_sigmarh_peso1}
\begin{split}
\Big | \sigmarh(\freal_{t_{j}}, \mathcal J) - \sigmarh(\freal_{t_{j-1}}, \mathcal K) \Big| &~ = \bigg | \frac{d}{dw} \conv_K \freal_{t_j,\e}(w) - \frac{d}{dw} \conv_K \freal_{t_{j-1},\e}(w) \bigg| \\
\text{(by Proposition \ref{P_estim_diff_conv})} &~ \leq  \bigg \| \frac{d}{dw} \freal_{t_j} - \frac{d}{dw} \freal_{t_{j-1}} \bigg\|_{L^\infty(K)} \\
\text{(by \eqref{E_diff_freal_peso}}) &~ \leq \big \| D^3_{wwv} f \big \|_{L^\infty} \big|v_{h(j)}\big|M(t_j,s,s')[p,p'].
\end{split}
\end{equation}
A similar computation yields
\begin{equation}
\label{E_diff_sigmarh_peso2}
\Big | \sigmarh(\freal_{t_{j}}, \mathcal K') - \sigmarh(\freal_{t_{j-1}}, \mathcal J') \Big| \leq \big \| D^3_{wwv} f \big\|_{L^\infty} \big|v_{h(j)}\big|M(t_j,s,s')[p,p'].
\end{equation}
Using \eqref{E_diff_sigmarh_peso1}, \eqref{E_diff_sigmarh_peso2} and the inductive assumption, we obtain
\begin{equation*}
\begin{split}
\sigmarh(\freal_{t_{j}}, \mathcal J) - \sigmarh(\freal_{t_{j}}, \mathcal J')
= &~ \Big[ \sigmarh(\freal_{t_{j}}, \mathcal J) - \sigmarh(\freal_{t_{j-1}}, \mathcal K)\Big] \\
&~ + \Big[ \sigmarh(\freal_{t_{j-1}}, \mathcal K) - \sigmarh(\freal_{t_{j-1}}, \mathcal K')\Big] \\
&~ + \Big[ \sigmarh(\freal_{t_{j-1}}, \mathcal K') - \sigmarh(\freal_{t_{j}}, \mathcal J')\Big] \\
\leq &~ \pi(t_{j-1}, s,s')[p,p'] + 2 \big \| D^3_{wwv} f \big \|_{L^\infty} \big|v_{h(j)}\big|M(t_j,s,s')[p,p'] \\
= &~ \pi(t_j,s,s')[p,p']. \qedhere
\end{split}
\end{equation*}
\end{proof}

\begin{lemma}
\label{L_peso_non_cambia_se_esterne}
Let $t_j$ be a fixed time. Let $s<s'$ be two waves, divided in the real solution at time $t_j$, but which have already interacted. Let $p,p' \in \mathcal I(t_j,s,s')$, $p \leq s < s' \leq p'$. Then for each $r,r' \in \mathcal I(t_j,s,s') = \mathcal I(t_j,p,p')$,
\[
\pi(t_j,s,s')[r,r'] = \pi(t_j,p,p')[r,r'].
\]
\end{lemma}
\begin{proof}
By the second part of Proposition \ref{P_partition_restr}, $\mathcal I(t_j,s,s') = \mathcal I(t_j,p,p')$ and $\mathcal P(t_j,s,s') = \mathcal P(t_j,p,p')$. The conclusion follows just observing that in the definition of $\pi(t_j,s,s'), \pi(t_j,p,p')$ only the partitions $\mathcal P(t_j,s,s'), \mathcal P(t_j,p,p')$ are used.
\end{proof}

\begin{theorem}
\label{T_decreasing_without_denominator}
Let $(t_j,x_j)$ be a positive interaction point. Let $\mathcal L$, $\mathcal R$ be the two wavefronts (considered as sets of waves) interacting in $(t_j,x_j)$, $\mathcal L < \mathcal R$. It holds
\[
\begin{split}
\sigmarh (\freal_{t_{j-1}}, \mathcal L) - & \sigmarh (\freal_{t_{j-1}}, \mathcal R) |\mathcal L| |\mathcal R| \\
\leq &~ \sum_{\substack{(s,s') \in \mathcal L \times \mathcal R \\ (s,s') \ \rm{already} \\ \rm{interacted}}} \pi(t_{j-1},s,s')[s,s']|s||s'| \quad +  \sum_{\substack{(s,s') \in \mathcal L \times \mathcal R \\ (s,s') \  \rm{never} \\ \rm{interacted}}} \|D^2_{ww} f\|_{L^\infty} \Big(|\mathcal L | + |\mathcal R| \Big) |s||s'|.
\end{split}
\]
\end{theorem}

\begin{proof}
First let us introduce some useful tools. For any rectangle $\mathcal C:=\tilde{\mathcal L} \times \tilde{\mathcal R} \subseteq \mathcal L \times \mathcal R$, define (see Figure \ref{fig:figura40}):
\[
\Phi_0 (\mathcal C) := 
\begin{cases}
\emptyset, & \mathcal C = \emptyset, \\
\Big[\tilde{\mathcal L} \cap \mathcal I(t_{j-1}, \max \tilde{\mathcal L}, \min \tilde{\mathcal R}) \Big] \times \Big[\tilde{\mathcal R} \cap \mathcal I(t_{j-1}, \max \tilde{\mathcal L}, \min \tilde{\mathcal R})\Big], & \max \tilde{\mathcal L}, \min \tilde{\mathcal R} \text{ already interacted,} \\
\{(\max \tilde{\mathcal L}, \min \tilde{\mathcal R} )\}, & \max \tilde{\mathcal L}, \min \tilde{\mathcal R} \text{ never interacted},
\end{cases}
\]

\[
\Phi_1 (\mathcal C) := 
\begin{cases}
\emptyset, 
& \mathcal C = \emptyset, \\
\Big[\tilde{\mathcal L} \cap \mathcal I(t_{j-1}, \max \tilde{\mathcal L}, \min \tilde{\mathcal R}) \Big] \times \Big[\tilde{\mathcal R} \setminus \mathcal I(t_{j-1}, \max \tilde{\mathcal L}, \min \tilde{\mathcal R})\Big],
& \max \tilde{\mathcal L}, \min \tilde{\mathcal R} \text{ already interacted,} \\
\{\max \tilde{\mathcal L}\} \times \Big[\tilde{\mathcal R} \setminus \{\min \tilde{\mathcal R}\}\Big],
& \max \tilde{\mathcal L}, \min \tilde{\mathcal R} \text{ never interacted},
\end{cases}
\]

\[
\Phi_2 (\mathcal C) := 
\begin{cases}
\emptyset, 
& \mathcal C = \emptyset, \\
\Big[\tilde{\mathcal L} \setminus \mathcal I(t_{j-1}, \max \tilde{\mathcal L}, \min \tilde{\mathcal R}) \Big] \times \Big[\tilde{\mathcal R} \setminus \mathcal I(t_{j-1}, \max \tilde{\mathcal L}, \min \tilde{\mathcal R})\Big],
& \max \tilde{\mathcal L}, \min \tilde{\mathcal R} \text{ already interacted,} \\
\Big[\tilde{\mathcal L} \setminus \{\max \tilde{\mathcal L}\}\Big] \times \Big[\tilde{\mathcal R} \setminus \{\min \tilde{\mathcal R}\}\Big],  
& \max \tilde{\mathcal L}, \min \tilde{\mathcal R} \text{ never interacted},
\end{cases}
\]

\[
\Phi_3 (\mathcal C) := 
\begin{cases}
\emptyset, 
& \mathcal C = \emptyset, \\
\Big[\tilde{\mathcal L} \setminus \mathcal I(t_{j-1}, \max \tilde{\mathcal L}, \min \tilde{\mathcal R}) \Big] \times \Big[\tilde{\mathcal R} \cap \mathcal I(t_{j-1}, \max \tilde{\mathcal L}, \min \tilde{\mathcal R})\Big], 
& \max \tilde{\mathcal L}, \min \tilde{\mathcal R} \text{ already interacted,} \\
\Big[\tilde{\mathcal L} \setminus \{\max \tilde{\mathcal L}\}\Big] \times \{\min \tilde{\mathcal R}\},
& \max \tilde{\mathcal L}, \min \tilde{\mathcal R} \text{ never interacted}.
\end{cases}
\]
Clearly $\Big\{\Phi_0(\mathcal C), \Phi_1(\mathcal C), \Phi_2(\mathcal C), \Phi_3(\mathcal C) \Big\}$ is a disjoint partition of $\mathcal C$.

\begin{figure}
  \begin{center}
    \includegraphics[height=8cm,width=10cm]{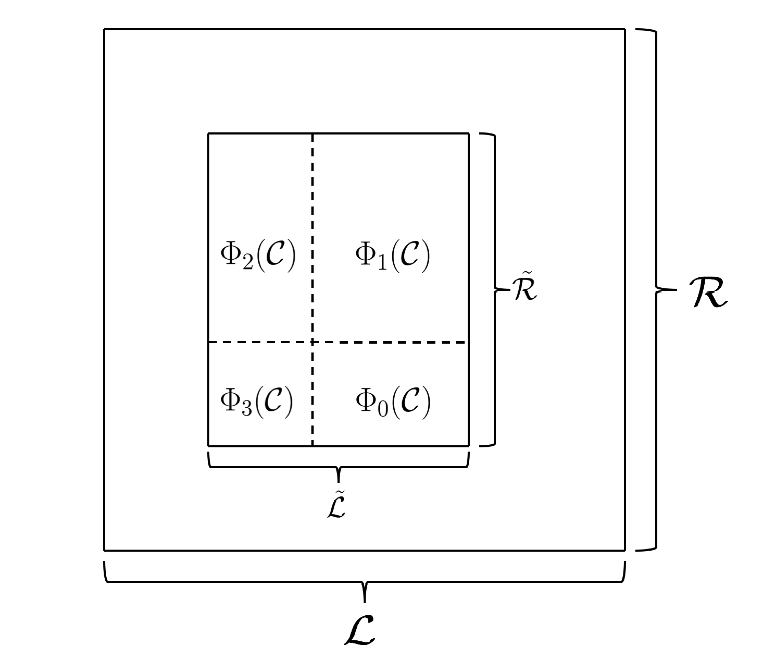}
    \caption{Partition of $\tilde{\mathcal L} \times \tilde{\mathcal R}$}
    \label{fig:figura40}
    \end{center}
\end{figure}

\noindent Denote by $\Delta \sigmarh (\freal_{t_{j-1}}, \mathcal C) := \sigmarh(\freal_{t_{j-1}}, \tilde{\mathcal L}) - \sigmarh(\freal_{t_{j-1}}, \tilde{\mathcal R})$ the difference in speed assigned to the first and the second edge of $\mathcal C$ by the effective flux function at time $t_{j-1}$. Set $|\mathcal C| := |\mathcal L| |\mathcal R| = \card(\mathcal C) \e^2$. By conservation it holds $\Delta \sigmarh (\freal_{t_{j-1}}, \mathcal C) |\mathcal C| = \sum_{a=0}^3 \Delta \sigmarh (\freal_{t_{j-1}}, \Phi_a( \mathcal C)) |\Phi_a(\mathcal C)|$.

\noindent For any set $A$, denote by $A^{<\N}$ the set of all finite sequences taking values in $A$. We assume that $\emptyset \in A^{<\N}$ and it is called \emph{the empty sequence}. There is a natural ordering $\unlhd$ on $A^{<\N}$: given $\alpha, \beta \in A^{<\N}$,
\[
\alpha \unlhd \beta \Longleftrightarrow \text{$\beta$ is obtained from $\alpha$ by adding a finite sequence.}
\]
A subset $D \subseteq A^{<\N}$ is called a \emph{tree} if for any $\alpha, \beta \in A^{<\N}$, $\alpha \unlhd \beta$, if $\beta \in D$, then $\alpha \in D$. 

Define a map $\Psi: \{1,2,3\}^{<\N} \longrightarrow 2^{\mathcal L \times \mathcal R}$, by setting
\[
\begin{split}
\Psi_\alpha  = 
\begin{cases}
\mathcal L \times \mathcal R, & \text{ if } \alpha = \emptyset, \\
\Phi_{a_n} \circ \cdots \circ \Phi_{a_1}(\mathcal L \times \mathcal R), & \text{ if }\alpha = (a_1, \dots, a_n) \in \{1,2,3\}^{<\N} \setminus \{\emptyset\}.
\end{cases}
\end{split}
\]
For $\alpha \in \{1,2,3\}^{<\N}$, let $\mathcal L_\alpha, \mathcal R_\alpha$ defined by the relation $\Psi_\alpha = \mathcal L_\alpha \times \mathcal R_\alpha$. 

\noindent Define a tree in $\{1,2,3\}^{<\N}$ setting $\alpha \in D \Longleftrightarrow \Psi_\alpha \neq \emptyset$. See Figure \ref{fig:figura39}.

\begin{figure}
  \begin{center}
    \includegraphics[height=8cm,width=9cm]{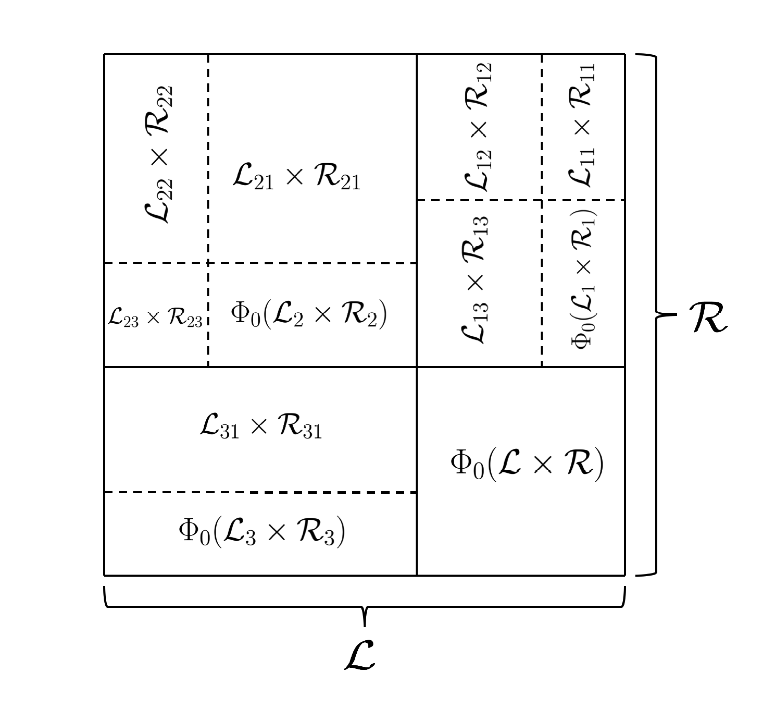}
    \caption{Example of partition of $\mathcal L \times \mathcal R$ using the tree $D$}
    \label{fig:figura39}
    \end{center}
\end{figure}

The idea of the proof is to show that for each $\alpha \in D$, on the rectangle $\Psi_\alpha$ it holds
\[
\begin{split}
\Delta \sigmarh (\freal_{t_{j-1}}, \Psi_\alpha)|\Psi_\alpha| 
\leq 
\sum_{\substack{(s,s') \in \Psi_\alpha \\ (s,s') \text{ already} \\ \text{interacted}}} \pi(t_{j-1},s,s')[s,s']|s||s'| + \sum_{\substack{(s,s') \in \Psi_\alpha \\ (s,s') \text{ never} \\ \text{interacted}}} \|D^2_{ww} f\|_{L^\infty} \Big(|\mathcal L | + |\mathcal R| \Big)|s||s'|.
\end{split}
\]
The conclusion will follow just considering that $\emptyset \in D$ and $\Psi_{\emptyset} = \mathcal L \times \mathcal R$. We need the following two lemmas.

\begin{lemma}
\label{L_incastro_2}
For any $\alpha \in D$, if $\max \mathcal L_\alpha, \min \mathcal R_\alpha$ have already interacted at time $t_{j-1}$, then the partition $\mathcal P(t_{j-1}, \max \mathcal L_\alpha, \min \mathcal R_\alpha)$ of $\mathcal I(t_{j-1}, \max \mathcal L_\alpha, \min \mathcal R_\alpha)$ can be restricted to 
\[
\mathcal L_\alpha \cap \mathcal I(t_{j-1}, \max \mathcal L_\alpha, \min \mathcal R_\alpha)
\]
and to
\[
\mathcal R_\alpha \cap \mathcal I(t_{j-1}, \max \mathcal L_\alpha, \min \mathcal R_\alpha).
\]
\end{lemma}

\begin{proof}
If $\alpha = \emptyset$ the proof is an easy consequence of Proposition \ref{P_divise_partizione_implica_divise_realta}. Thus assume $\alpha = \beta a $ for some $a \in \{1,2,3\}$, $\beta \in D$. By simplicity assume $a = 1$, the other cases being similar. In this case it is not difficult to see that
\[
\mathcal L_\alpha \cap \mathcal I(t_{j-1}, \max \mathcal L_\alpha, \min \mathcal R_\alpha) = \mathcal L \cap \mathcal I(t_{j-1}, \max \mathcal L_\beta, \min \mathcal R_\beta) \cap \mathcal I(t_{j-1}, \max \mathcal L_\alpha, \min \mathcal R_\alpha).
\]
We have that $\mathcal P(t_{j-1}, \max \mathcal L_\alpha, \min \mathcal R_\alpha)$ can be restricted both to $\mathcal L \cap \mathcal I(t_{j-1}, \max \mathcal L_\alpha, \min \mathcal R_\alpha)$ by Proposition \ref{P_divise_partizione_implica_divise_realta} and to $\mathcal I(t_{j-1}, \max \mathcal L_\beta, \min \mathcal R_\beta) \cap \mathcal I(t_{j-1}, \max \mathcal L_\alpha, \min \mathcal R_\alpha)$ by Proposition \ref{P_partition_restr}, since 
\begin{equation}
\label{E_beta_less_alpha}
\max \mathcal L_\alpha \leq \max \mathcal L_\beta < \min \mathcal R_\beta \leq \min \mathcal R_\alpha.
\end{equation}
Hence $\mathcal P(t_{j-1}, \max \mathcal L_\alpha, \min \mathcal R_\alpha)$ can be restricted to  $\mathcal L_\alpha \cap \mathcal I(t_{j-1}, \max \mathcal L_\alpha, \min \mathcal R_\alpha)$. 

\noindent Similarly,
\[
\mathcal R_\alpha \cap \mathcal I(t_{j-1}, \max \mathcal L_\alpha, \min \mathcal R_\alpha) = \mathcal R \cap \Big(\mathcal I(t_{j-1}, \max \mathcal L_\alpha, \min \mathcal R_\alpha) \setminus \mathcal I(t_{j-1}, \max \mathcal L_\beta, \min \mathcal R_\beta) \Big).
\] 
As before, $\mathcal P(t_{j-1}, \max \mathcal L_\alpha, \min \mathcal R_\alpha)$ can be restricted both to $\mathcal R \cap \mathcal I(t_{j-1}, \max \mathcal L_\alpha, \min \mathcal R_\alpha)$ (by Proposition \ref{P_divise_partizione_implica_divise_realta}) and to $\mathcal I(t_{j-1}, \max \mathcal L_\alpha, \min \mathcal R_\alpha) \setminus \mathcal I(t_{j-1}, \max \mathcal L_\beta, \min \mathcal R_\beta)$ (by Proposition \ref{P_partition_restr} and \eqref{E_beta_less_alpha}); thus it can be restricted also to $\mathcal R_\alpha \cap \mathcal I(t_{j-1}, \max \mathcal L_\alpha, \min \mathcal R_\alpha)$. 
\end{proof}

\begin{lemma}
\label{L_estim_phi_zero_psi_alpha}
For each $\alpha \in D$, if $\max \mathcal L_\alpha, \min \mathcal R_\alpha$ have already interacted at time $t_{j-1}$, then on $\Phi_0(\Psi_\alpha)$ it holds
\[
\begin{split}
\Delta \sigmarh(\freal_{t_{j-1}}, \Phi_0 (\Psi_\alpha)) & |\Phi_0 (\Psi_\alpha)| 
\leq 
\sum_{(s,s') \in \Phi_0(\Psi(\alpha))} \pi(t_{j-1},s,s')[s,s']|s||s'|.\end{split}
\]
\end{lemma}

\begin{proof}
By definition of $\Phi_0$,
\[
\Phi_0 (\Psi_\alpha) = \Big[\mathcal L_\alpha \cap \mathcal I(t_{j-1}, \max \mathcal L_\alpha, \min \mathcal R_\alpha) \Big] \times \Big[\mathcal R_\alpha \cap \mathcal I(t_{j-1}, \max \mathcal L_\alpha, \min \mathcal R_\alpha)\Big].
\]
By previous lemma, 
\[
\begin{split}
\mathcal P(t_{j-1}, \max \mathcal L_\alpha, \min \mathcal R_\alpha)|_{\mathcal L_\alpha \cap \mathcal I(t_{j-1}, \max \mathcal L_\alpha, \min \mathcal R_\alpha)} = &~ \Big\{ \mathcal J_1, \dots, \mathcal J_L \Big\}, \\
\mathcal P(t_{j-1}, \max \mathcal L_\alpha, \min \mathcal R_\alpha)|_{\mathcal R_\alpha \cap \mathcal I(t_{j-1}, \max \mathcal L_\alpha, \min \mathcal R_\alpha)} = &~ \Big\{ \mathcal K_1, \dots, \mathcal K_N \Big\}.
\end{split}
\]
Hence,
\[
\begin{split}
\Delta \sigmarh(\freal_{t_{j-1}}, \Phi_0 (\Psi_\alpha)) |\Phi_0 (\Psi_\alpha)| \\
\text{(by conservation)} &~ = \sum_{l=1}^L \sum_{n=1}^N \Delta \sigmarh(\freal_{t_{j-1}}, \mathcal J_l \times \mathcal K_n)  |\mathcal J_l \times \mathcal K_n| \\
&~ = \sum_{l=1}^L \sum_{n=1}^N \Delta \sigmarh(\freal_{t_{j-1}}, \mathcal J_l \times \mathcal K_n)  \sum_{(s,s') \in \mathcal J_l \times \mathcal K_n} |s||s'| \\
\text{(by Lemma \ref{L_estim_interval_of_partition})}&~ \leq \sum_{l=1}^L \sum_{n=1}^N \sum_{(s,s') \in \mathcal J_l \times \mathcal K_n} \pi(t_{j-1}, \max \mathcal L_\alpha, \min \mathcal R_\alpha)[s,s']|s||s'| \\
&~ \leq \sum_{(s,s') \in \Phi_0(\Psi_\alpha)} \pi(t_{j-1}, \max \mathcal L_\alpha, \min \mathcal R_\alpha)[s,s']|s||s'| \\
\text{(by Lemma \ref{L_peso_non_cambia_se_esterne})} &~ \leq \sum_{(s,s') \in \Phi_0(\Psi_\alpha)} \pi(t_{j-1}, s,s')[s,s']|s||s'|.  \qedhere
\end{split}
\]
\end{proof}

\noindent \textbf{Conclusion of the proof of Theorem \ref{T_decreasing_without_denominator}.}
As said before, to conclude the proof of the theorem it is sufficient to show that for each $\alpha \in D$, on $\Psi_\alpha$ it holds
\begin{equation}
\label{E_estim_psi_alpha}
\begin{split}
\Delta \sigmarh(\freal_{t_{j-1}}, \Psi_\alpha)|\Psi_\alpha| \leq 
\sum_{\substack{(s,s') \in \Psi_\alpha \\ (s,s') \text{ already} \\ \text{interacted}}} \pi(t_{j-1},s,s')[s,s']|s||s'| + \sum_{\substack{(s,s') \in \Psi_\alpha \\ (s,s') \text{ never} \\ \text{interacted}}} \|D^2_{ww} f\|_{L^\infty} \Big(|\mathcal L | + |\mathcal R| \Big) |s||s'|.
\end{split}
\end{equation}
This is proved by (inverse) induction on the tree $D$. If $\alpha \in D$ is a leaf of the tree (i.e. $\alpha a \notin D$ for each $a \in \{1,2,3\}$), then $\Psi_\alpha = \Phi_0(\Psi_\alpha)$. If $\max \mathcal L_\alpha, \min \mathcal R_\alpha$ have never interacted at time $t_{j-1}$, then $\Psi_\alpha = \Phi_0(\Psi_\alpha) = \{(\max \mathcal L_\alpha, \min \mathcal R_\alpha)\}$ and 
inequality \eqref{E_estim_psi_alpha} follows from Mean Value Theorem; if $\max \mathcal L_\alpha, \min \mathcal R_\alpha$ have already interacted at time $t_{j-1}$, then each wave in $\mathcal L_\alpha$ have interacted with any wave in $\mathcal R_\alpha$ and thus inequality \eqref{E_estim_psi_alpha} is a consequence of Lemma \ref{L_estim_phi_zero_psi_alpha}.

Now take $\alpha \in D$, $\alpha$ not a leaf. Then 
$\Phi_1(\Psi_\alpha) = \Psi_{\alpha 1}$, $\Phi_2(\Psi_\alpha) = \Psi_{\alpha 2}$, $\Phi_3(\Psi_\alpha) = \Psi_{\alpha 3}$ and
\begin{equation*}
\begin{split}
\Delta \sigmarh(\freal_{t_{j-1}}, \Psi_\alpha)|\Psi_\alpha| \\
\text{(by conservation)} = &~  \Delta \sigmarh(\freal_{t_{j-1}}, \Phi_0(\Psi_\alpha)) |\Phi_0(\Psi_\alpha)|  +  \sum_{a=1}^{3} \Delta \sigmarh(\freal_{t_{j-1}}, \Phi_a(\Psi_\alpha)) |\Phi_a(\Psi_\alpha)| \\
= &~  \Delta \sigmarh(\freal_{t_{j-1}}, \Phi_0(\Psi_\alpha)) |\Phi_0(\Psi_\alpha)|  +  \sum_{a=1}^{3} \Delta \sigmarh(\freal_{t_{j-1}}, \Psi_{\alpha a}) |\Psi_{\alpha a}| \\
\text{(by Lemma \ref{L_estim_phi_zero_psi_alpha})} \leq &~  \sum_{(s,s') \in \Phi_0(\Psi_\alpha)} \pi(t_{j-1},s,s')[s,s']|s||s'| + \sum_{a=1}^{3} \Delta \sigmarh(\freal_{t_{j-1}}, \Psi_{\alpha a}) |\Psi_{\alpha a}| \\
\text{(by inductive assum} & \text{ption)} \\
\leq &~ 
\sum_{(s,s') \in \Phi_0(\Psi_\alpha)} \pi(t_{j-1},s,s')[s,s']|s||s'| \\
&~ + \sum_{a=1}^3 \bigg(\sum_{\substack{(s,s') \in \Psi_{\alpha a} \\ (s,s') \text{ already} \\ \text{interacted}}} \pi(t_{j-1},s,s')[s,s']|s||s'| \\
&~  \quad \ + \sum_{\substack{(s,s') \in \Psi_{\alpha a} \\ (s,s') \text{ never} \\ \text{interacted}}} \|D^2_{ww} f\|_{L^\infty} \Big(|\mathcal L | + |\mathcal R| \Big) |s||s'| \bigg) \\
= &~ \sum_{\substack{(s,s') \in \Psi_\alpha \\ (s,s') \text{ already} \\ \text{interacted}}} \pi(t_{j-1},s,s')[s,s']|s||s'| + \sum_{\substack{(s,s') \in \Psi_\alpha \\ (s,s') \text{ never} \\ \text{interacted}}} \|D^2_{ww} f\|_{L^\infty} \Big(|\mathcal L | + |\mathcal R| \Big) |s||s'|,
\end{split}
\end{equation*}
thus concluding the proof of the theorem.
\end{proof}

\begin{corollary}
\label{W_decreasing}
For any interaction point $(t_j,x_j)$, it holds
\begin{equation*}
\sum_{s \in \W(t_j)} |\sigma(t_j, s) - \sigma(t_{j-1}, s)||s| \leq 2 \big[ \fQ(t_{j-1}) - \fQ(t_j) \big].
\end{equation*}
\end{corollary}

By direct inspection of the proof one can verify that the constant $2$ is sharp.

\begin{proof}
As said at the beginning of this section, we assume w.l.o.g. that all the waves in $\W(t_j,x_j)$ are positive. Let $\mathcal L$, $\mathcal R$ be the two wavefronts (considered as sets of waves) interacting in $(t_j,x_j)$, $\mathcal L < \mathcal R$.
With standard arguments, observing that the waves which change speed after the interaction are those in $\mathcal L \cup \mathcal R$, one can see that
\begin{equation*}
\begin{split}
\sum_{s \in \W(t_j)} |\sigma(t_j, s) - \sigma(t_{j-1}, s)||s| 
&~ = 2 \frac{|\sigmarh (\freal_{t_{j-1}}, \mathcal L) - \sigmarh (\freal_{t_{j-1}}, \mathcal R)|  |\mathcal L| |\mathcal R|}{|\mathcal L| + |\mathcal R| } \\
&~ = 2 \frac{\Big(\sigmarh (\freal_{t_{j-1}}, \mathcal L) - \sigmarh (\freal_{t_{j-1}}, \mathcal R)\Big)  |\mathcal L| |\mathcal R|}{|\mathcal L| + |\mathcal R| }, \\
\end{split}
\end{equation*}
where the last equality is justified by the fact that, since $\mathcal L$ and $\mathcal R$ are interacting and $\mathcal L < \mathcal R$, then $\sigmarh (\freal_{t_{j-1}}, \mathcal L) > \sigmarh (\freal_{t_{j-1}}, \mathcal R)$. Hence, using Theorem \ref{T_decreasing_without_denominator}, we obtain
\[
\begin{split}
\sum_{s \in \W(t_j)} & |\sigma(t_j, s) - \sigma(t_{j-1}, s)||s| \\
&~ = 2 \frac{\Big(\sigmarh (\freal_{t_{j-1}}, \mathcal L) - \sigmarh (\freal_{t_{j-1}}, \mathcal R)\Big)  |\mathcal L| |\mathcal R|}{|\mathcal L| + |\mathcal R| } \\
&~ \leq \frac{2}{|\mathcal L| + |\mathcal R|} \Bigg[ \sum_{\substack{(s,s') \in \mathcal L \times \mathcal R \\ \text{$(s,s')$ already} \\ \text{interacted}}} \pi(t_{j-1},s,s')[s,s']|s||s'| \quad +  \sum_{\substack{(s,s') \in \mathcal L \times \mathcal R \\ \text{$(s,s')$ never} \\ \text{interacted}}} \|D^2_{ww} f\|_{L^\infty} \Big(|\mathcal L | + |\mathcal R| \Big) |s||s'| \Bigg] \\
&~ \leq 2 \Bigg[ \sum_{\substack{(s,s') \in \mathcal L \times \mathcal R \\ \text{$(s,s')$ already} \\ \text{interacted}}} \frac{\pi(t_{j-1},s,s')[s,s']|s||s'|}{\hat w(s') - (\hat w(s) - \e)} \quad +  \sum_{\substack{(s,s') \in \mathcal L \times \mathcal R \\ \text{$(s,s')$ never} \\ \text{interacted}}} \|D^2_{ww} f\|_{L^\infty} |s||s'| \Bigg] \\
&~ = 2 \sum_{(s,s') \in \mathcal L \times \mathcal R} \mathfrak q(t_{j-1},s,s')|s||s'| \\
&~ = 2 \Big[ \fQ(t_{j-1}) - \fQ(t_j)  \Big],
\end{split}
\]
which is what we wanted to get.
\end{proof}

\subsection{\texorpdfstring{Increasing part of $\fQ$}{Increasing part of Q}}
\label{Ss_Q_increase}

This section is devoted to prove inequality \eqref{W_increase}, more precisely we will prove the following theorem.

\begin{theorem}
\label{W_increasing}
If $(t_j,x_j)$ is a transversal interaction point, then
\begin{equation*}
\fQ(t_j) - \fQ(t_{j-1}) \leq 6 \log(2) \|D^3_{wwv}f\|_{L^\infty} |v_{h(j)}| |\W(t_j,x_j)| \TV(w(0,\cdot)),
\end{equation*}
where $|v_{h(j)}|$  is the strength of the wavefront of the first family involved in the transversal interaction at time $t_j$.
\end{theorem}

\begin{proof}
Assume for simplicity that $\W(t_j,x_j) = [p_1,p_2]$ and waves in $\W(t_j, x_j)$ are positive. First of all, let us split the quantity we want to estimate as follows. Recall that given $s<s'$ in $\W(t_j)$, their weight can increase only if they are divided both before and after the transversal interaction and have already interacted before the transversal interaction.
\begin{equation}
\label{E_increasing_part_splitted}
\begin{split}
\fQ(t_j) - \fQ(t_{j-1}) = &~ \sum_{\substack{s,s' \in \W(t_j) \\ s<s'}} \Big[\mathfrak q(t_j,s,s') - \mathfrak q(t_{j-1},s,s')\Big]|s||s'| \\
= &~ \sum_{s < \W(t_j,x_j)} \sum_{\substack{s' \in \W(t_j,x_j) \\ s' \text{ already interacted} \\ \text{with $s$}}} \Big[\mathfrak q(t_j,s,s') - \mathfrak q(t_{j-1},s,s')\Big]|s||s'| \\
&~ + \sum_{s < \W(t_j,x_j)} \sum_{\substack{s' > \W(t_j,x_j) \\ s' \text{ already interacted} \\ \text{with $s$}}} \Big[\mathfrak q(t_j,s,s') - \mathfrak q(t_{j-1},s,s')\Big]|s||s'| \\
&~ + \sum_{s \in \W(t_j,x_j)} \sum_{\substack{s' > \W(t_j,x_j) \\ s' \text{ already interacted} \\ \text{with $s$}}} \Big[\mathfrak q(t_j,s,s') - \mathfrak q(t_{j-1},s,s')\Big]|s||s'|. \\
\end{split}
\end{equation}

Let us begin with the estimate on the first term of the summation. Fix $s < \W(t_j,x_j)$ and assume that $s$ has interacted with $p_1$. Set
\[
r_s:= \max\Big\{p \in \W(t_j,x_j) \ \Big| \ \text{$p$ has interacted with $s$ at time $t_{j-1}$}\Big\}.
\]
We need the following lemma.

\begin{lemma}
\label{L_partition_dep_on_s}
There exists a partition $\mathcal P_s = \{\mathcal K_1, \dots, \mathcal K_K \}$ of $[p_1,r_s]$, with $\mathcal K_k < \mathcal K_{k+1}$ such that for any $s' \in [p_1,r_s]$, if $\mathcal K_{k(s')}$ is the element of the partition $\mathcal P_s$ containing $s'$, it holds
\[
M(t_j,s,s')[s,s'] = \sum_{k=1}^{k(s')} \mathcal K_k.
\]
\end{lemma}
\noindent The remarkable point in this lemma is the fact that the partition $\mathcal P_s$ depends only on $s$ and not on $s'$.

\begin{proof}
Define $N \in \N$ and a finite sequence $(s_n)_{n = 1, \dots, N+1}$ as follows. Set $s_1 := p_1$. Now assume to have defined $s_n \in [p_1,r_s]$; set $s_{n+1} := \max \mathcal I(t_{j-1},s,s_n) +1$. If $s_{n+1} \leq r_s$, keep on the recursive procedure, otherwise set $N:=n$ and stop the procedure. Clearly $s_1 < \dots < s_N$. 

Now observe that for each $n = 1, \dots N$, the partition $\mathcal P(t_{j-1},s,s_n)$ can be restricted to $\mathcal I(t_{j-1},s,s_n) \cap [s_n,s_{n+1}-1]$. Namely for $n=1$ this follows from Proposition \ref{P_divise_partizione_implica_divise_realta}, while for $n \geq 2$, it is a consequence of Proposition \ref{P_partition_restr}, just observing that 
\[
\mathcal I(t_{j-1},s,s_n) \cap [s_n,s_{n+1}-1] = \mathcal I(t_{j-1},s,s_n) \setminus \mathcal I(t_{j-1},s,s_{n-1})
\]
and $s < s_{n-1} \leq s_n$.
Hence we can set
\[
\mathcal P_s := \bigcup_{n=1}^N \mathcal P(t_{j-1},s,s_n)|_{\mathcal I(t_{j-1},s,s_n) \cap [s_n,s_{n+1}-1]}
\]
Observe also that for each $s' \in [s_n, s_{n+1} -1] \subseteq \mathcal I(t_{j-1},s,s_n)$, by the second part of Proposition \ref{P_partition_restr}, it holds $\mathcal I(t_{j-1},s,s') = \mathcal I(t_{j-1},s,s_n)$ and $\mathcal P(t_{j-1},s,s') = \mathcal P(t_{j-1},s,s_n)$.

Hence for any $s' \in  [s_n, s_{n+1} -1]$, denoting by $\mathcal J_s, \mathcal J_{s'}$ the elements of $\mathcal P(t_{j-1},s,s')$ containing $s,s'$ respectively,  we have
\[
\begin{split}
M(t_j,s,s')[s,s'] &~ =  \sum_{\substack{\mathcal J \in \mathcal P(t_{j-1},s,s') \\ \mathcal J_s \leq \mathcal J \leq \mathcal J_{s'} \\ \mathcal J \subseteq \W(t_j,x_j)}} |\mathcal J |
= \sum_{\substack{\mathcal J \in \mathcal P(t_{j-1},s,s_n) \\ \mathcal J_s \leq \mathcal J \leq \mathcal J_{s'} \\ \mathcal J \subseteq \W(t_j,x_j)}} |\mathcal J |
= \sum_{\substack{\mathcal J \in \mathcal P(t_{j-1},s,s_n) \\ \mathcal J < s_n \\ \mathcal J \subseteq \W(t_j,x_j)}} |\mathcal J | + 
\sum_{\substack{\mathcal J \in \mathcal P(t_{j-1},s,s_n) \\ \mathcal J \subseteq [s_n, s_{n+1}-1] \\ \mathcal J \leq \mathcal J_{s'}}} |\mathcal J | \\
&~ = \Big| [s_1,s_n-1] \Big| + \sum_{\substack{\mathcal K \in \mathcal P_s \\ \mathcal K \subseteq [s_n, s_{n+1}-1] \\ \mathcal K \leq \mathcal K_{k(s')}}} |\mathcal K | 
= \sum_{\substack{\mathcal K \in \mathcal P_s \\ \mathcal K < \{s_n\}}} |\mathcal K | + \sum_{\substack{\mathcal K \in \mathcal P_s \\ \mathcal K \subseteq [s_n, s_{n+1}-1] \\ \mathcal K \leq \mathcal K_{k(s')}}} |\mathcal K |
= \sum_{k=1}^{k(s')}
|\mathcal K |. \qedhere
\end{split}
\]
\end{proof}

Now for fixed $s < \W(t_j,x_j)$, $s$ already interacted with $p_1$, consider the partition $\mathcal P_s = \{\mathcal K_1, \dots, \mathcal K_K \}$ of $[p_1,r_s]$, with $\mathcal K_k < \mathcal K_{k+1}$ constructed in previous lemma and, as before, for any $s' \in [p_1,r_s]$ denote by $\mathcal K_{k(s')}$ the element of the partition $\mathcal P_s$ containing $s'$. We have 
\begin{equation*}
\begin{split}
\sum_{s' \in [p_1,r_s]} \Big[\mathfrak q(t_j,s,s') - &\mathfrak q(t_{j-1},s,s')\Big]  |s||s'| \\
= &~ 2\|D^3_{wwv}f\|_{L^\infty} |v_{h(j)}||s| \sum_{s' = p_1}^{r_s} \frac{1}{\hat w(s') - (\hat w(s) - \e)}
M(t_j,s,s')[s,s']|s'| \\
\text{(by Lemma \ref{L_partition_dep_on_s})} = &~ 2\|D^3_{wwv}f\|_{L^\infty} |v_{h(j)}||s| \sum_{s' = p_1}^{r_s} 
\frac{1}{\hat w(s') - (\hat w(s) - \e)} |s'|
\sum_{k=1}^{k(s')} |\mathcal K_k| \\
= &~ 2\|D^3_{wwv}f\|_{L^\infty} |v_{h(j)}||s| \sum_{l=1}^{K} \sum_{s' \in \mathcal K_l} \frac{1}{\hat w(s') - (\hat w(s) - \e)}  |s'|
\sum_{k=1}^{l} |\mathcal K_k| \\
= &~ 2\|D^3_{wwv}f\|_{L^\infty} |v_{h(j)}||s| \sum_{k=1}^{K} \sum_{l=k}^{K} \sum_{s' \in \mathcal K_l}
\frac{1}{\hat w(s') - (\hat w(s) - \e)} |s'| |\mathcal K_k| \\
\leq &~ 2\|D^3_{wwv}f\|_{L^\infty} |v_{h(j)}||s| 
\sum_{k=1}^K |\mathcal K_k| \sum_{s' =p_1}^{r_s} \frac{|s'|}{\hat w(s') - (\hat w(s) - \e)} \\
\leq &~ 2\|D^3_{wwv}f\|_{L^\infty} |v_{h(j)}||s| |\W(t_j,x_j)| \int_{\hat w(p_1) -\e}^{\hat w(p_2)} \frac{dw'}{w' - (\hat w(s) - \e)}.
\end{split}
\end{equation*}
Now we sum over all $s < \W(t_j,x_j)$ which have interacted with $p_1$.
\begin{equation*}
\begin{split}
\sum_{\substack{s < \W(t_j,x_j) \\ \text{$s$ interacted with $p_1$}}} 
\sum_{s' \in [p_1,r_s]} &\Big[\mathfrak q(t_j,s,s') - \mathfrak q(t_{j-1},s,s')\Big] |s||s'| \\
\leq &~
2\|D^3_{wwv}f\|_{L^\infty} |v_{h(j)}| |\W(t_j,x_j)| \sum_{\substack{s < \W(t_j,x_j) \\ \text{$s$ interacted with $p_1$}}} 
|s|  \int_{\hat w(p_1) -\e}^{\hat w(p_2)} \frac{dw'}{w' - (\hat w(s) - \e)} \\
\leq &~ 2\|D^3_{wwv}f\|_{L^\infty} |v_{h(j)}| |\W(t_j,x_j)| 
\int_a^{\hat w(p_1) - \e} \int_{\hat w(p_1) - \e}^{\hat w(p_2)} \frac{dw'dw}{w'-w},
\end{split}
\end{equation*}
where $a:= \hat w\bigg(\min \Big\{s \ \Big | \ \text{$s$ has interacted with $p_1$} \Big\} \bigg) -\e$. An easy computation shows that 
\begin{equation}
\label{E_log2}
\int_a^\xi \int_\xi^b \frac{dw'dw}{w'-w} \leq \log(2) (b-a).
\end{equation}
Hence,
\begin{equation}
\label{E_increasing_first}
\begin{split}
\sum_{s < \W(t_j,x_j)} 
\sum_{s' \in \W(t_j,x_j)} &\Big[\mathfrak q(t_j,s,s') - \mathfrak q(t_{j-1},s,s')\Big] |s||s'| \\
\leq &~ 2 \log(2) \|D^3_{wwv}f\|_{L^\infty} |v_{h(j)}| |\W(t_j,x_j)| \TV(w(0,\cdot)).
\end{split}
\end{equation}

A similar computation holds for the third term in the summation \eqref{E_increasing_part_splitted} and gives
\begin{equation}
\label{E_increasing_third}
\begin{split}
\sum_{s \in \W(t_j,x_j)} 
\sum_{\substack{s' > \W(t_j,x_j) \\ \text{$s'$ already interacted with $s$}}} &\Big[\mathfrak q(t_j,s,s') - \mathfrak q(t_{j-1},s,s')\Big] |s||s'| \\
\leq &~ 2 \log(2) \|D^3_{wwv}f\|_{L^\infty} |v_{h(j)}| |\W(t_j,x_j)| \TV(w(0,\cdot)).
\end{split}
\end{equation}

The estimate on the second term of the summation in \eqref{E_increasing_part_splitted} is similar, but easier. 
\[
\begin{split}
\sum_{s < \W(t_j,x_j)} & \sum_{\substack{s' > \W(t_j,x_j) \\ s' \text{ already interacted} \\ \text{with $s$}}}  \Big[\mathfrak q(t_j,s,s') - \mathfrak q(t_{j-1},s,s')\Big]|s||s'| \\
\leq &~ \sum_{s < \W(t_j,x_j)} \sum_{\substack{s' > \W(t_j,x_j) \\ s' \text{ already interacted} \\ \text{with $s$}}} 2 \|D^3_{wwv}f\|_{L^\infty} |v_{h(j)}| |\W(t_j,x_j)| \frac{|s||s'|}{\hat w(s') - (\hat w(s) - \e)} \\
\leq &~ 2 \|D^3_{wwv}f\|_{L^\infty} |v_{h(j)}| |\W(t_j,x_j)| \int_a^{\hat w(p_1) - \e} \int_{\hat w(p_2)}^b \frac{dw'dw}{w'-w},
\end{split}
\]
where $a := \hat w(\min \mathcal M) - \e$, $b:= \hat w(\max \mathcal M)$ and $\mathcal M$ is the maximal monotone interval in which $\W(t_j, x_j)$ is contained. Thus, using again \eqref{E_log2}, 
\begin{equation}
\label{E_increasing_second}
\begin{split}
\sum_{s < \W(t_j,x_j)} \sum_{\substack{s' > \W(t_j,x_j) \\ s' \text{ already interacted} \\ \text{with $s$}}} & \Big[\mathfrak q(t_j,s,s') - \mathfrak q(t_{j-1},s,s')\Big]|s||s'| \\
\leq &~ 2 \log(2) \|D^3_{wwv}f\|_{L^\infty} |v_{h(j)}| |\W(t_j,x_j)| \TV(w(0,\cdot)).
\end{split}
\end{equation}
Summing up inequalities \eqref{E_increasing_first}, \eqref{E_increasing_third}, \eqref{E_increasing_second}, we conclude the proof.
\end{proof}

\begin{corollary}
It holds
\[
\sum_{\substack{t_j \ \mathrm{ transversal} \\ \mathrm{interaction}}}\fQ(t_j) - \fQ(t_{j-1}) \leq 6 \log(2) \|D^3_{wwv}f\|_{L^\infty} \TV(w(0,\cdot))^2 \TV(v(0,\cdot)).
\]
\end{corollary}

\begin{proof}
By previous theorem, 
\[
\fQ(t_j) - \fQ(t_{j-1}) \leq 6 \log(2) \|D^3_{wwv}f\|_{L^\infty} \TV(w(0,\cdot)) \Big[\Qtrans(t_{j-1}) - \Qtrans(t_j) \Big].
\]
Using Proposition \ref{P_Qtrans} we get the thesis.
\end{proof}


\begin{thebibliography}{99}








\bibitem{anc_mar_10} F. Ancona, A. Marson, 
\textit{A locally quadratic Glimm functional and sharp convergence rate of the Glimm scheme for nonlinear hyperbolic systems}, 
Arch. Ration. Mech. Anal. \textbf{196} (2010), 455-487.


\bibitem{anc_mar_11_CMP} F. Ancona, A. Marson, 
\textit{Sharp Convergence Rate of the Glimm Scheme for General Nonlinear Hyperbolic Systems},
Comm. Math. Phys. \textbf{302} (2011), 581-630.


\bibitem{anc_mar_11_DCDS} F.~Ancona, A.~Marson, 
\textit{On the Glimm Functional for General Hyperbolic Systems}, 
Discrete and Continuous Dynamical Systems, Supplement, (2011), 44-53.


\bibitem{bai_bre_97} P.~Baiti, A.~Bressan,
\textit{The semigroup generated by a Temple class system with large data}, Diff. Integr. Equat., \textbf{10} (1997), 401–418.


\bibitem{bia_00} S.~Bianchini, \textit{The semigroup generated by a Temple class system with non-convex flux function}, Diff. Integr. Equat.,
\textbf{13(10-12)} 2000, 1529–1550.


\bibitem{bia_03} S.~Bianchini, 
\textit{Interaction Estimates and Glimm Functional for General Hyperbolic Systems}, 
Discrete and Continuous Dynamical Systems \textbf{9} (2003), 133-166. 

%

%


\bibitem{bia_mod_13} S.~Bianchini, S.~Modena, 
\textit{On a quadratic functional for scalar conservation laws},
preprint SISSA (2013).

\bibitem{bia_mod_14} S.~Bianchini, S.~Modena, 
in preparation.

\bibitem{bre_00} A. Bressan, 
\textit{Hyperbolic Systems of Conservation Laws. The One Dimensional Cauchy problem}, 
Oxford University Press, (2000).


%






\bibitem{gli_65} J.~Glimm, 
\textit{Solutions in the Large for Nonlinear Hyperbolic Systems of Equations}, 
Comm. Pure Appl. Math. \textbf{18} (1965), 697-715.


\bibitem{hua_jia_yan_10} J.~Hua, Z.~Jiang, T.~Yang, 
\textit{A New Glimm Functional and Convergence Rate of Glimm Scheme for General Systems of Hyperbolic Conservation Laws}, Arch. Rational Mech. Anal. 
\textbf{196} (2010), 433-454.


\bibitem{hua_yang_10} J. Hua, T. Yang, 
\textit{A Note on the New Glimm Functional for General Systems of Hyperbolic Conservation Laws}, 
Math. Models Methods Appl. Sci. \textbf{20(5)} (2010), 815-842.

%

%


\bibitem{temple_83} B. Temple,
\textit{Systems of conservation laws with invariant submanifolds},
Trans. Amer. Math. Soc. \textbf{280} (1983), 781-795.











\end{thebibliography}
\end{document}